\documentclass[10pt,a4paper,reqno]{amsart}
\usepackage[cp1250]{inputenc}
\usepackage[english]{babel}
\usepackage{color}
\usepackage[colorlinks,linkcolor=blue,citecolor=red]{hyperref}
\usepackage{amssymb}
\usepackage{amsmath}
\usepackage{lmodern}
\usepackage[IL2]{fontenc}
\usepackage{url}
\usepackage{graphicx}
\usepackage{standalone}

\newtheorem{defi}{\bf D\scriptsize EFINITION \normalsize}
\newtheorem{theorem}[defi]{\bf T\scriptsize HEOREM \normalsize}
\newtheorem{lm}[defi]{\bf L\scriptsize EMMA \normalsize}
\newtheorem{dk}[defi]{\bf C\scriptsize OROLLARY \normalsize}
\newtheorem{rem}[defi]{\bf R\scriptsize EMARK \normalsize}
\newtheorem{exa}[defi]{\bf E\scriptsize XAMPLE \normalsize}
\newtheorem{pro}{\bf P\scriptsize ROBLEM \normalsize}
\newtheorem{prop}[defi]{\bf P\scriptsize ROPOSITION \normalsize}
\newtheorem{no}{\bf N\scriptsize OTE \normalsize}

\def\kopr{\hfill\raisebox{3pt}{\framebox{$\star$}}}
\newenvironment{example}{\begin{exa}\rm}{$\kopr$\end{exa}}
\newenvironment{remark}{\begin{rem}\rm}{$\kopr$\end{rem}}
\newenvironment{definition}{\begin{defi}\rm}{\end{defi}}

\newenvironment{lemma}{\begin{lm}\it}{\end{lm}}
\newenvironment{corollary}{\begin{dk}\it}{\end{dk}}
\newenvironment{proposition}{\begin{prop}\it}{\end{prop}}

\newcommand{\nadsebou}[2]{\begin{array}{c} #1 \\ #2 \end{array}}

\newcommand{\zav}[1]{\left( #1 \right)}
\newcommand{\szav}[1]{\left\{ #1 \right\}}

\newcommand{\imag}{{\rm i}}

\newcommand{\abs}[1]{\left| #1 \right|}
\renewcommand\Re{\operatorname{Re}}
\renewcommand\Im{\operatorname{Im}}

\newcommand{\deter}[1]{\det\zav{#1}}
\newcommand{\tr}{{\rm tr}\,}
\newcommand{\trunc}[1]{\left\lfloor #1 \right\rfloor}
\allowdisplaybreaks
\begin{document}
\title{Coval description of the boundary of a numerical range and the secondary values of a matrix}
 \author{Petr Blaschke}
 \thanks{}
\address{ Mathematical Institute, Silesian University in Opava, Na Rybnicku 1, 746 01 Opava, Czech Republic}
\email{Petr.Blaschke@math.slu.cz}
\begin{abstract}
The boundary of a numerical range of a finite matrix is always a nice curve (algebraic, closed and simple), but the equation it satisfies is often very complicated. We will show that, furthermore, there is no hope of describing these curves in terms of distances from the eigenvalues -- as the dimension~2, where the numerical range is just an ellipse, would suggest. But, as we will show, there is a remarkably simple ``coval'' description in terms of distances to \textit{tangent lines}. Provided that one measures these distances not only to the eigenvalues but also to additional points, the most important of which are \textit{secondary values} -- which we will define and describe their algebraic and geometric properties.
\end{abstract}
\maketitle

\section{Introduction}
For a $n$-dimensional square complex matrix $A\in \mathbb{C}^{n\times n}$ its \textit{numerical range} $W(A)\subset \mathbb{C}$ is the set
$$
W(A):=\szav{\frac{u^\star A u}{u^\star u};\ u\in \mathbb{C}^n\setminus \szav{0}},
$$
which is invariant under unitary transformations, i.e.
$$
W(A)=W(U^\star A U),\qquad U^\star U=UU^\star =I,
$$  
and contains the eigenvalues
$$
\sigma(A)\subset W(A).
$$
The result of the famous Hausdorff-Toeplitz theorem states that $W(A)$ is always convex, and therefore its boundary $\partial W(A)$ is a simple closed curve. It is these curves that we are interested in. See Figure~\ref{gallery} for a small gallery.
\begin{figure}[ht] 
\centering 
\includegraphics[scale=0.5,trim=1cm 4.0cm 1cm 5cm,clip]{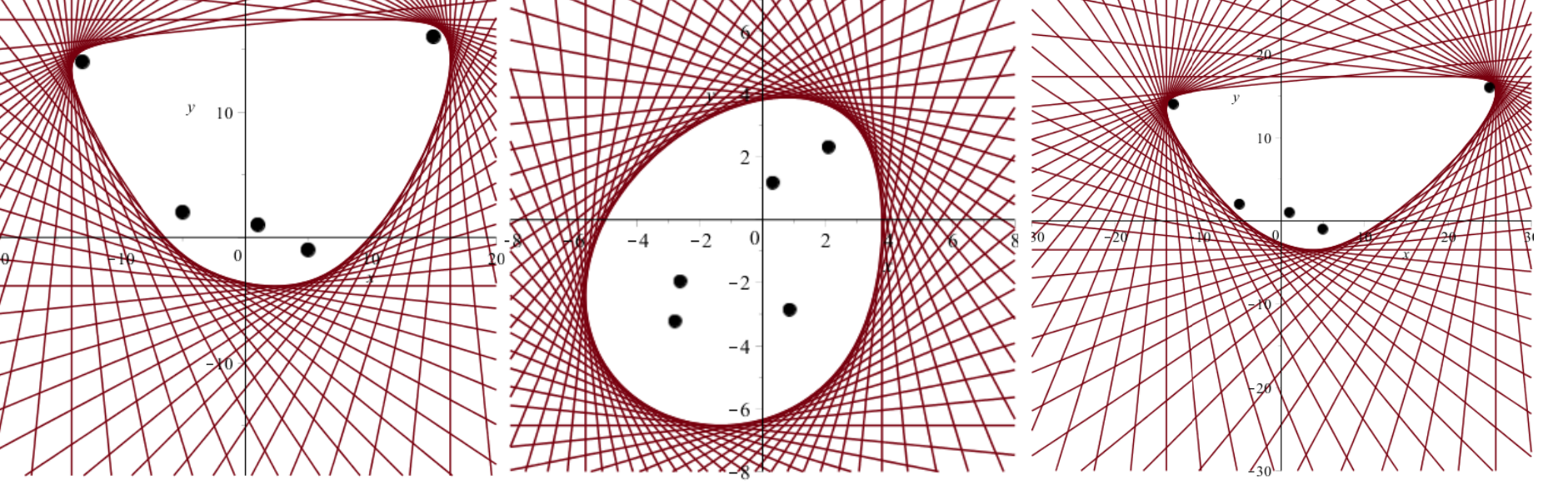}
\caption{Gallery of numerical ranges for various matrices. Here, the set is depicted as a locus of lines. The dots represents the eigenvalues.}
\label{gallery}
\end{figure}

Despite the elegant beauty of these curves, there is no simple algebraic description for them. It is known that $W(A)$ is the convex hull of an algebraic curve $C(A)$ called Kippenhahn curve of the \textit{class} number $n$ and the degree $\leq n(n-1)$. But this algebraic curve can be quite complicated. 

For instance, consider
$$
A:=\left( \begin {array}{ccccc} 1+i&5&-4&50&-5+i\\ \noalign{\medskip}0&5
-i&2-6\,i&5&25\\ \noalign{\medskip}0&0&-13+14\,i&-1+i&14
\\ \noalign{\medskip}0&0&0&15+16\,i&36\\ \noalign{\medskip}0&0&0&0&-5+
2\,i\end {array} \right).
$$
It can be shown that the corresponding Kippenhahn curve $C(A)$ is given by the following polynomial of degree 20
\begin{align*}
C(A)&=4696285754466050602115637473045464\,{x}^{20}+
9517781493448847593551172849185272\,{x}^{19}y\\
&+49331785388107432468516318314527096\,{x}^{18}{y}^{2}+
71723836435184918473021867224615504\,{x}^{17}{y}^{3}\\
&+\dots -53038022858170458767920887262741022050829752284474220\,x\\
&-137173718629830309296618124748301692267121186443002464\,y\\
&+85877448668995016962695243211402395052991126935768556,
\end{align*}
where the three dots represent an additional 220 terms. It is clear that a more efficient description is needed.
In dimension 2 there is a very simple one due to Toeplitz: 
\begin{proposition}\label{n2prop}(\cite{Toeplitz})
The boundary of a numerical range for $2\times 2$ matrix is an ellipse with foci located at the eigenvalues of $A$, i.e. it holds
\begin{equation}\label{n2case}
\forall z\in\partial W(A):\qquad |z-a|+|z-b|=C,
\end{equation}
where $a,b$ are eigenvalues of $A$ and $C>0$ is some constant.
\end{proposition}
\begin{remark}
We are going to give an elegant alternative proof of (\ref{n2case}) later.
\end{remark}
This is a very neat description, since it not only elegantly describes $\partial W(A)$, but also gives meaning to the eigenvalues.

Is there a similar description in higher dimensions? For the $3\times 3$ case, one conjecture might be that the boundary points $z$ of $W(A)$, where $A\in \mathbb{C}^{3\times 3}$, satisfy
\begin{equation}\label{trialtrillipse}
|z-a|+|z-b|+|z-c|\stackrel{?}{=}L,
\end{equation}
where $a,b,c\in \sigma(A)$ and $L>0$. In other words that a \textit{trillipse} is the answer.

Unfortunately, this is not true. Consider the following matrix.
$$
A_0:=\left( \begin {array}{ccc} 1+i&2+i&1\\ \noalign{\medskip}0&5-i&3+i
\\ \noalign{\medskip}0&0&-1+2\,i\end {array} \right).
$$
Its corresponding Kippenhahn curve reads:
\begin{align}
C(A_0):=\nonumber&60752\,{x}^{6}+233872\,{x}^{5}y+613524\,{x}^{4}{y}^{2}+1025696\,{x}^{3
}{y}^{3}+1206840\,{x}^{2}{y}^{4}+893200\,x{y}^{5}+384500\,{y}^{6}\\
\nonumber &-566848\,{x}^{5}-2408352\,y{x}^{4}-5636272\,{x}^{3}{y}^{2}-8074720\,{y}
^{3}{x}^{2}-6812400\,x{y}^{4}-2728000\,{y}^{5}\\
\nonumber &+1891120\,{x}^{4}+
9243072\,y{x}^{3}+18686944\,{x}^{2}{y}^{2}+18747840\,x{y}^{3}+7597600
\,{y}^{4}-4248260\,{x}^{3}\\
\nonumber&-16941432\,{x}^{2}y-23257860\,x{y}^{2}-
11055800\,{y}^{3}+4897364\,{x}^{2}+12045760\,xy+7823900\,{y}^{2}\\
\label{n3example}&-1615440\,x-1538400\,y-483825=0.
\end{align}
There is only one free parameter in (\ref{trialtrillipse}) -- the constant $L$, and after fixing it with a single boundary point of $W(A_0)$ we will see that the curves do not match. See Figure~\ref{trillipsefig}. 
\begin{figure}[ht] 
\centering 
\includegraphics[scale=0.4,trim=2cm 9cm 2cm 2cm,clip]{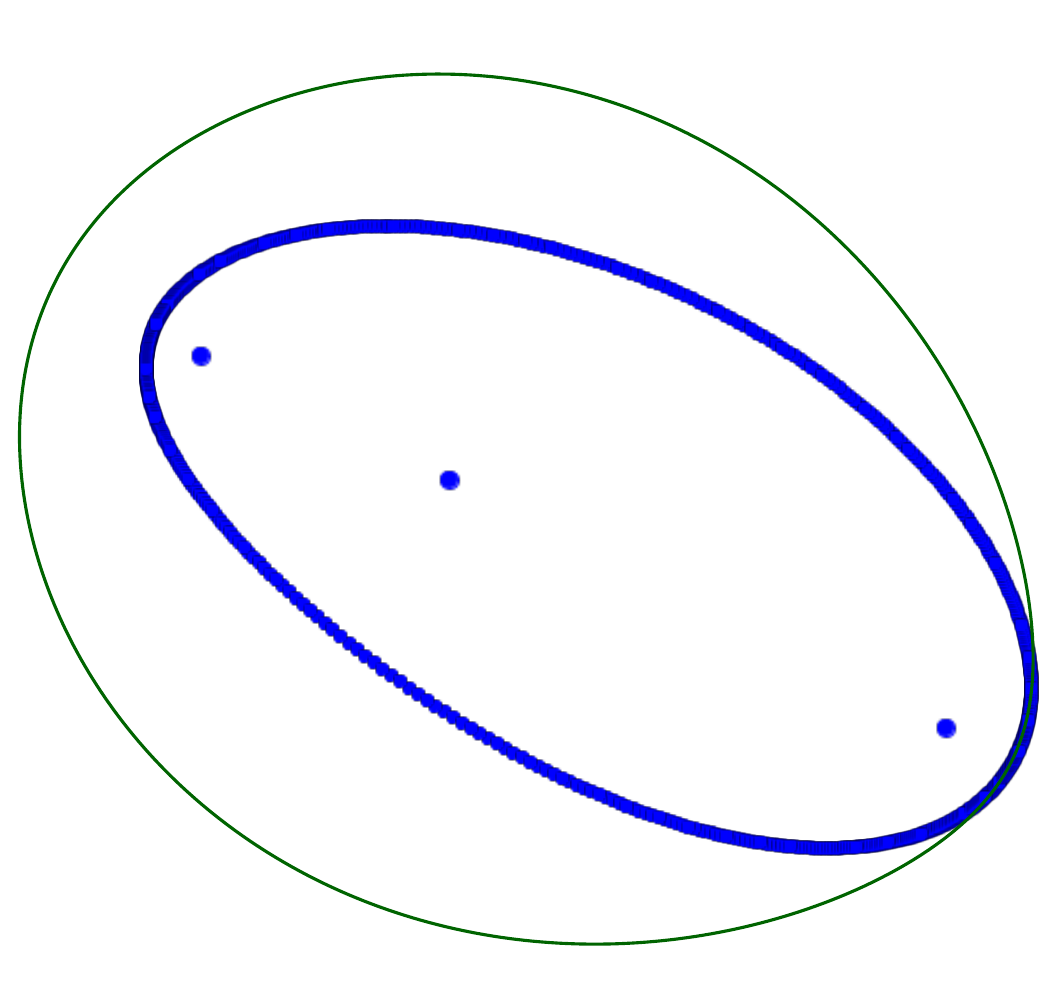}
\caption{Numerical range of $A_0$ with its eigenvalues and the trial trillipse.}
\label{trillipsefig}
\end{figure}
Hence numerical ranges of $3\times 3$ matrices are not trillipses in general. What about other ovals?

Let $r(a)$ denote the distance from a point on a given curve to point $a$, which we will call the ``pole''. By an oval we mean any curve such that there is an (algebraic) relation between the distances to different poles, similar to an ellipse $r(a)+r(b)=L$ or a trillipse $r(a)+r(b)+r(c)=L$. The precise definition of an oval will be given in Section~\ref{ovalsec}. 

The notion of \textit{focus} can be generalized to arbitrary algebraic curves. It is known that the only (real) foci of the Kippenhahn curve of an $n$-dimensional matrix are the eigenvalues (see \cite[6.1]{WuGau}). The eigenvalues are therefore natural candidates for poles. Does the curve $\partial W(A)$ have a nice oval description with eigenvalues as poles?

We are going to show the following 
\begin{proposition}\label{P2}
Let $A\in\mathbb{C}^{3\times 3}$ with distinct linearly independent eigenvalues $a,b,c\in\mathbb{C}$. There exists a polynomial $f$ in three variables such that the Kippenhahn curve of $A$ satisfies
\begin{equation}
f(r(a),r(b),r(c))=0.
\end{equation}
\end{proposition}
Sadly, this polynomial $f$ is not pretty in general.
Returning back to our previous example $A_0$, its oval description is even lengthier than its the Kippenhahn curve (\ref{n3example})!
\begin{align}
\nonumber &{\frac {16614359891\, {r(a)} ^{10}}{600000
}}+{\frac {100955575169\, {r(a)} ^{8}}{
720000}}+{\frac {7039216247\, {r(a)} ^{6}
}{27000}}+{\frac {264989089\, {r(a)} ^{4}
}{25920}}-{\frac {759925079\, {r(a)} ^{2}
}{1620}}\\
\label{n3oval}&+\dots-{\frac {1156665161\, {r(c)} ^{2
} {r(a)} ^{6} r \left( b \right)^{2}}{810000}}-{\frac {168829903\, r \left( c \right)^{2} {r(a)} ^{4} r \left( 
b \right)^{6}}{82012500}}-{\frac {175168565}{486}}=0,
\end{align}
where $a=1+\imag$, $b=5-\imag$, $c=2\imag-1$ are the eigenvalues of $A_0$ 
and the dots represents additional 70 terms.

In conclusion, numerical ranges of even 3 by 3 matrices are ovals but not, in general, simple ones.
\begin{remark}
At least if only the eigenvalues are considered. The question remains whether the above formula can be simplified by choosing other poles than just the eigenvalues or by adding more poles to the mix. As we will see, it is often advantageous to choose more poles than the absolute minimum. However, we will not attempt to answer this question here. 
\end{remark}

There are other types of ``ovals'', but instead of points, the distances to given poles are measured to \textit{tangents}. Let us denote by $p(a)$ the so-called \textit{pedal coordinate} -- the distance of a tangent of a given curve to the point $a$, called ``pole''. Since the quantity $p(a)$ is in a sense dual to the quantity $r(a)$, any curve that obeys an (algebraic) relation between distances to a finite number of poles will be called ``coval''. The precise definition is given in Section~\ref{covalsec}.

A simplest non-trivial coval is an ellipse, which can be described as a locus of lines such that the \textit{product} of distances to two foci is constant. Specifically
$$
p(a)p(b)=L^2,
$$
where $a,b$ are the foci of the ellipse and $L$ is its semi-minor axis. This geometrical gem known from antiquity (see \cite{lockwood}) is relatively forgotten; the author only recently learned of it from a lecture by Vladimir Dragovi\'c, who proves it in \cite[Lemma 3.1.]{Dragovic}. In fact, this description of an ellipse, for which we give an alternative proof in Proposition~\ref{ellipsecovalP}, sparked the whole idea for this paper.  

An ellipse is a numerical range of a matrix in dimension 2. What about dimension 3? As we have seen, the Kippenhahn curve for our example matrix $A_0$ (\ref{n3example}) enjoys an algebraic equation with $28$ terms, or the oval description with 78 terms (\ref{n3oval}). But the corresponding coval representation has only \textit{two} terms: 
\begin{equation}\label{n3coval}
p(a)p(b)p(c)=4p(\chi),\qquad \chi:=\frac{5+14\imag}{16},
\end{equation} 
involving only distances of tangents to the eigenvalues $a,b,c$ and to the additional number $\chi$ which, interestingly, is a member of the numerical range, i.e. $\chi\in W(A_0)$. We call this number \textit{the secondary value} of $A_0$.

This is not a happy coincidence, since -- as we will show in Corollary~\ref{3by3cor} -- the Kippenhahn curves for \textit{all} 3 by 3 matrices enjoy a two-term coval representation involving only eigenvalues and an additional number as poles exactly in the same form as (\ref{n3coval}).

In general, it seems that a coval description is natural for the boundary curve of a numerical range in all dimensions. In Proposition~\ref{numrangecopolar} we will show that the curve $\partial W(A)$ satisfies
\begin{equation}
\deter{p(A)}=0,
\end{equation} 
where $p(A)$ is a matrix analogue of the pedal coordinate $p(a)$. The precise definition will be given in (\ref{pAdef}). 

The main result of this paper is the following:
\begin{theorem}\label{maintint}
For any matrix $A\in\mathbb{C}^{n\times n}$, $n\geq 2$ it holds
\begin{equation}\label{covalpaint}
\deter{p(A)}=\sum_{k=0}^{\trunc{\frac{n}{2}}}\gamma_k\prod_{j=1}^{n-2k}p\zav{\lambda^{(k+1)}_{j}}.
\end{equation}
\end{theorem}
See Theorem~\ref{maint} for all the details.

This result shows that the boundary curve of the numerical range of an $n\times n$ complex matrix can be described by a $\trunc{n/2}+1$ term coval equation involving $(\trunc{n/2}+1)(n-\trunc{n/2})$ poles -- \textit{provided} that we allow some of the poles to be "at infinity". The definition of a pole at infinity is given in Section~\ref{covalsec}. 

Now -- \textit{typically} -- all the numbers $\lambda^{(k)}_j$ are finite. In fact, the author is yet to find a specific example of a matrix with a pole at infinity. Such a matrix would have to have a dimension of at least 5. 
 
The first $n$ poles $\lambda_j^{(1)}$ are always finite and coincide with the eigenvalues of $A$. The next $n-2$ poles $\lambda^{(2)}_j$ exists only for non-normal matrices and are also always finite and seem to play an important role in the subject. We call them \textit{secondary values} of $A$. Similarly to eigenvalues they are solutions of a simple algebraic equation and their centroid is always a member of the numerical range. 

Furthermore, in Theorem~\ref{T2} we show that a necessary condition for a given matrix $A$ to decompose into a direct sum of a number and a lower dimensional matrix, is that a secondary value coincides with an eigenvalue. We also show some geometric properties of secondary values in dimensions 3 and 4. Finally, some open problems are presented at the end.

\section{Ovals}\label{ovalsec}

\begin{definition}For any $z,a\in\mathbb{C}$ denote the distance between $z$ and $a$ by
$$
r(a):=|z-a|.
$$
\end{definition}
\begin{remark}
The \textit{function} $r(a)$ (depending on $z$) can be thought as a \textit{coordinate}, the same way the polar distance $r$ is. If $a=0$ we usually equate
$$
r(0)\equiv r.
$$
\end{remark}
\begin{definition}
A plane curve $\gamma:[a,b]\to \mathbb{C}$ is an \textit{oval} if there exists a function $f:U\subset\mathbb{R}^m\to \mathbb{R}$ and $m$ complex numbers $a_1,\dots, a_m$ called ``poles'' such that
\begin{equation}
\forall z\in \gamma([a,b]):\qquad f(r(a_1),\dots, r(a_m))=0.
\end{equation}
\end{definition}
\begin{remark}
In other words, an \textit{oval} is the locus of points such that a relation between distances to a number of poles holds.
\end{remark}
Here is few examples of well-known ovals. In the following list, the arguments of $r$ functions  are complex numbers; the remaining constants are real.
\begin{align}
&r(a)=0, & \text{Point.}\\
&r(a)=R, & \text{Circle of radius $R$.}\\
&r(a)+r(b)=C, & \text{Ellipse.}\\
&r(a)-r(b)=C, & \text{Hyperbola.}\\
&r(a)=r(b), & \text{Line.}\\
&\frac{r(a)}{r(b)}=C, & \text{Circle.}\\
&r(a)r(-a)=a, & \text{Lemniscate of Bernoulli.}\\
&r(a)r(b)=C, & \text{Cassini oval.}\\
&\alpha r(a)+\beta r(b)=C, & \text{Cartesian Oval.}\\
&r(a)+r(b)+r(c)=C, & \text{Trillipse.}
\end{align}
See e.g. \cite{Lawrence, Zwikker}.
\begin{example}
The oval notation permits one to make deep geometrical implications based on a simple algebra. For instance, it is known that a curve which keeps a ratio of distances to two points constant is a circle.

What is the curve which keeps the ratio of distance to two \textit{circles}? Distance to a circle with the center at $a$ and radius $R_a$ is given by $|r(a)-R_a|$. The 
equation for our constant ratio oval is therefore
$$
\frac{r-R}{r(a)-R_a}=\alpha.
$$
Rearranging yields
$$
r-\alpha r(a)=R-\alpha R_a.
$$ 
A Cartesian oval.
\end{example}

\begin{remark}
Since the usual Cartesian coordinates $x,y$ can be written
\begin{equation}\label{Cartesiantooval}
x=\frac{r(-1)^2-r(1)^2}{4},\qquad y=\frac{r(-\imag)^2-r(\imag)^2}{4},
\end{equation}
it follows that every algebraic curve, i.e. a curve given implicitly by $p(x,y)=0$ where $p\in\mathbb{R}[x,y]$, is an oval with only four poles. We are going to show that the number of poles can be always reduced to three.
\end{remark}

\begin{lemma} (Principle of triangulation)
Let $a,b,c,d\in\mathbb{C}$ such that $a,b,c$ are not co-linear. Then
\begin{equation}\label{triangp}
\deter{\begin{array}{ccccc}
r(a)^2 & r(b)^2 & r(c)^2 & r(d)^2 & 1 \\
1 & 1 & 1 & 1 & 0 \\
\Re a & \Re b & \Re c & \Re d & 0 \\
\Im a & \Im b & \Im c & \Im d & 0 \\
|a|^2 & |b|^2 & |c|^2 & |d|^2 & 1 
\end{array}}=0,
\end{equation}
i.e. every distance $r(d)^2$ is an affine combination of three other distances. 
\end{lemma}
\begin{proof}
Exercise.
\end{proof}
\begin{remark}
If the poles $a,b,c$ are co-linear the equation (\ref{triangp}) become trivial.
\end{remark}
\begin{corollary}\label{C1}
Every algebraic curve has a three poles oval description, provided that the three poles are not co-linear.
\end{corollary}
\begin{remark}
Corollary~\ref{C1} proves Proposition~\ref{P2}.
\end{remark}
\begin{example}
Consider the matrix $A_0$ from the introduction
$$
A_0:=\left( \begin {array}{ccc} 1+i&2+i&1\\ \noalign{\medskip}0&5-i&3+i
\\ \noalign{\medskip}0&0&-1+2\,i\end {array} \right),
$$
and its Kippenhahn curve (\ref{n3example}).
Using formulas (\ref{Cartesiantooval}) we can replace algebraic equation (\ref{n3example}) with an oval with poles $1,-1,\imag,-\imag$. The Principle of triangulation (\ref{triangp}) further allows us to replace quantities
$
r(1)^2,r(-1)^2,r(\imag)^2,r(-\imag)^2
$
with
$
r(a)^2,r(b)^2,r(c)^2,
$
where $a=1+\imag$,$b=5-\imag$,$c=2\imag-1$ are the eigenvalues of $A_0$. Specifically, we get the following formulas
$$
x=\frac{13 r(c)^2-r(b)^2-27 r(a)^2+15}{90},\qquad y=\frac{r(b)^2-9r(a)^2-4 r(c)^2+12}{36}.
$$
Substituting these into (\ref{n3example}) thus gives an oval description of $\partial W(A_0)$ with eigenvalues as poles in the form (\ref{n3oval}), showing that Kippenhahn curves of 3 by 3 matrices are not simple ovals.  
\end{example}

\section{Covals}\label{covalsec}
\begin{definition}
For any piece-wise differentiable plane curve $\gamma: [\alpha,\beta]\in \mathbb{R}\to \mathbb{C}$ parameterized by arc-length $\gamma(s)=x+\imag y$, i.e. $\gamma'(s)=x'+\imag y'$: $x'^2+y'^2=1$, and for any complex number $a$, we define the following quantities
\begin{align}
p(a)&:=-(y-\Im a) x'+(x-\Re a) y', & \text{pedal coordinate,}\\
p_c(a)&:=(x-\Re a) x'+(y-\Im a) y', & \text{contrapedal coordinate.}\\
(x',y')&=:(\cos \theta, \sin \theta), & \text{$\theta$ tangential angle.}
\end{align}
\end{definition}
\begin{remark}
Geometrically, the pedal coordiante $p(a)$ is the distance of $a$ to the tangent of $\gamma$ at the point $(x,y)$ -- or rather the absolute value $|p(a)|$. The quantity $p(a)$ can be negative depending on the curve's orientation (for our purposes the orientation will be unimportant). 

In the same way, the absolute value of the contrapedal coordinate $p_c(a)$ is the distance of $a$ to the \textit{normal} instead of tangent. The contrapedal coordiante $p_c(a)$ can be also sometimes computed from the following equation
$$
p(a)^2+p_c(a)^2=r(a)^2,
$$
linking the three quantities together.

The remaining variable $\theta$ is the angle between the tangent and $x$-axis and its known in the literature as the tangential angle. Similarly to arc-length $s$ and curvature $\kappa$ this quantity is ``intrinsic'' to the curve in the sense that $\theta$ does not change performing translation or scaling. It does change when rotations are taken into the account but predictably. The tangential angle is used in the so-called Whewell coordinate system $(s,\theta)$. The relation ship to the other intrinsic coordinates is as follows
$$
\frac{\partial \theta}{\partial s}=\kappa.
$$
Another important differential relation between all the quantities we have introduced is the following.
\begin{equation}\label{diffrules}
r(a)\frac{\partial r(a)}{\partial s}=\frac{\partial p(a)}{\partial \theta}=p_c(a).
\end{equation}

\end{remark} 

\begin{definition}
We will call a plane curve $\gamma:[a,b]\to \mathbb{C}$ a \textit{coval} if there exists a function $f:U\subset\mathbb{R}^m\to \mathbb{R}$ and $m$ complex numbers $a_1,\dots, a_m$ called ``poles'' such that
\begin{equation}
\forall x+\imag y\in \gamma([a,b]):\qquad f(p(a_1),\dots, p(a_m))=0.
\end{equation}
\end{definition}
\begin{remark}
In other words, a \textit{coval} is the locus of points such that a relation between distances from given points to any tangent, holds. In this paper, we will mostly deal with the case when $f$ is a polynomial.
\end{remark}
We will now prove the coval representation of an ellipse.
\begin{proposition}\label{ellipsecovalP}
Any ellipse with foci at $a,b\in\mathbb{C}$ and the semi-minor axis $L$ satisfies
\begin{equation}\label{ellipsecoval}
p(a)p(b)=L^2,
\end{equation}
that is the product of distances from foci to any tangent is constant.
\end{proposition}
\begin{figure}[h] 
\centering 
\includegraphics[scale=1,trim=0.1cm 0.1cm 0.1cm 0.1cm,clip]{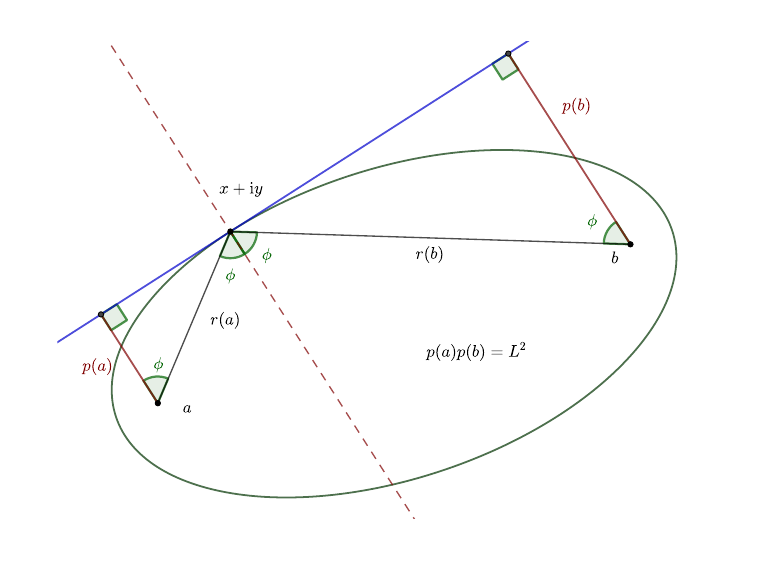}
\caption{Coval description of an ellipse.}
\label{ellipsecovalfig}
\end{figure}
\begin{proof}
We are going to use three fundamental fact about ellipses. Assuming that $M$ is the major axis of our ellipse (and $2L$ is the minor axis) we have:
\begin{align}
r(a)+r(b)&=M, & &\text{Oval description,}\\
\frac{p(a)}{r(a)}&=\frac{p(b)}{r(b)}, & &\text{law of reflection,}\\
\frac{L^2}{p(a)^2}&=\frac{M}{r(a)}-1, & &\text{pedal equation.}
\end{align}
The first equation is the defining property. 

The second one refers to the ellipse's optical quality  -- a beam of light emitted from one focus gets reflected into the other one. Denoting the angle of incidence by $\phi$, it is easy to see that 
$$
\frac{p(a)}{r(a)}=\cos \phi =\frac{p(b)}{r(b)},
$$
by the law of reflection. See Figure~\ref{ellipsecovalfig}.

The pedal equation can be found e.g. \cite{Yates, Edwards}. It is also a fundamental property since, as explained in \cite{lockwood,Blaschke6}, the form of the pedal equation tells us that an ellipse is a solution of \textit{Kepler problem} -- i.e. it is the trajectory of a test particle under the influence of a gravitational force from point mass object (located in the point $a$). In this context, $M$ is proportional to the mass of the attracting body and $L$ is proportional to the particle's angular momentum.

Eliminating $r(a)$ and $r(b)$ from these equations yields the result. 
\end{proof}
Here is a short list of covals
\begin{align}
\nonumber \text{Equation} && &\text{a solution}\\
p(a)&=0 & &\text{every line passing } a.\\ 
p(a)&=R & &\text{a circle with radius $R$ and center $a$.}\\
p(a)p(b)&=L & &\text{an ellipse with foci $a$, $b$.}\\ 
(p(a)-p(c))^2&=p(c)^2-L^2 & &\text{an ellipse with focus $a$ and center $c$.}\\ 
\label{covalcardioid}27a^2p&=4p(a)^3 & &\text{Cardioid with the cusp at $0$ and center at $a$.}
\end{align}
\begin{remark}\label{remark10}
It is important to stress that a coval representation does not define the curve uniquely. 
Consider for some function $f$ a coval:
$$
f(p(a_1),\dots, p(a_m))=0.
$$
It is obvious that any tangent line of a solution is also a solution. Therefore this equation does not describe a single curve but additional solutions can be obtained attaching a tangent at some point(s).  The resulting patchwork curves are also differentiable and we will call them \textit{comets}. See Figure~\ref{ellipticalcometfig}.
\end{remark}
\begin{figure}[h] 
\centering 
\includegraphics[scale=1,trim=0.1cm 0.1cm 0.1cm 0.1cm,clip]{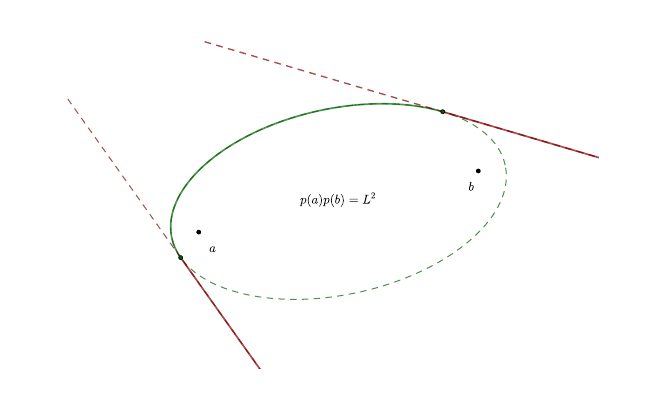}
\caption{Elliptical comet. Also a solution of $p(a)p(b)=L^2$.}
\label{ellipticalcometfig}
\end{figure}
\subsection{Singular solution}
Observe that for any $a\in\mathbb{C}$ we can write the pedal coordinate $p(a)$ as follows:
\begin{equation}\label{difsubs}
p(a)=(x-a_1)y'-(y-a_2)x'=\frac{(x-a_1)y'_x-(y-a_2)}{\sqrt{1+{y'_x}^2}}, \qquad y'_x:=\frac{{\rm d} y(x)}{{\rm d} x}.
\end{equation}
Substituting this into a coval equation
$$
f(p(a_1),\dots, p(a_m))=0
$$
we obtain a first order (in general \textit{nonlinear}) differential equation
$$
f(p(a_1),\dots, p(a_m))=\tilde f\zav{x,y,y'_x}=0.
$$
Eliminating $y'_x$ from the following set of equations
\begin{equation}\label{singsoleq}
\tilde f\zav{x,y,y'_x}=0,\qquad \frac{\partial\tilde f(x,y,y'_x)}{\partial y'_x}=0,
\end{equation}
we obtain what is known as \textit{singular solution} or \textit{discriminant curve} of the original equation. The singular solution has the property that touches any ``regular'' solutions. But the precise definition of a singular solution (and a regular one for that matter) is kind of difficult. See \cite{izumiya} for a discussion on the subject. It is these singular solutions that we are usually interested in.
\begin{example}
What is the singular solution of the coval
$$
p(a)=0?
$$
We have
$$
p(a)=\frac{(x-a_1)y'_x-(y-a_2)}{\sqrt{1+{y'_x}^2}}=0
$$
thus we can take
$$
\tilde f(x,y,y'_x)=\frac{(x-a_1)y'_x-(y-a_2)}{\sqrt{1+{y'_x}^2}}.
$$
Now
$$
\frac{\partial\tilde f}{\partial y'_x}=\frac{(x-a_1)+(y-a_2)y'_x}{\zav{1+{y'_x}^2}^\frac32}=0
$$
implies 
$$
y'_x=-\frac{x-a_1}{y-a_2}.
$$
Substituting this yields
$$
\sqrt{(x-a_1)^2+(y-a_2)^2}=0
$$
Therefore $x+\imag y=a$. A point. 
(A point is, indeed, a curve with the agreement that any line passing through this point is a tangent; we can think of this as a circle with zero radius).  
\end{example}

\begin{example}
Starting with the equation of an ellipse
$$
p(a)p(b)=L^2,\qquad a=a_1+\imag a_2,\ b=b_1+\imag b_2,\qquad L>0,
$$
we obtain the following differential equation
$$
\tilde f:=\zav{(x-a_1)y'_x-(y-a_2)}\zav{(x-b_1)y'_x-(y-b_2)}-L^2\zav{1+{y'_x}^2}=0.
$$
Differentiating with respect to $y'_x$ yields
$$
\frac{\partial \tilde f}{\partial y'_x}= 2y'_x\zav{(x-a_1)(x-b_1)-L^2}-(x-a_1)(y-b_2)-(x-b_1)(y-a_2).
$$
Eliminating $y'_x$ from $\tilde f$, $\frac{\partial \tilde f}{\partial y'_x}$ gives us the following singular solution:
$$
4\zav{(x-a_1)(x-b_1)-L^2}\zav{(y-a_2)(y-b_2)-L^2}=\zav{(x-a_1)(y-b_2)+(x-b_1)(y-a_2)}^2,
$$
which simplifies to a quadric:
\begin{align*}
&\zav{4L^2+(a_2-b_2)^2}x^2-2(a_2-b_2)(a_1-b_1)xy+\zav{4L^2+(a_1-b_1)^2}y^2\\
&-2\zav{2(a_1+b_1)L^2-(a_2-b_2)(a_1b_2-a_2b_2)}x
-2\zav{2(a_2+b_2)L^2+(a_1-b_1)(a_1b_2-a_2b_2)}y\\
&-4L^4+4(a_1b_1+a_2b_2)L^2+(a_1b_2-a_2b_1)^2=0,
\end{align*}
with negative discriminant
$$
\Delta:=-16L^2\zav{4L^2+(a_1-b_1)^2+(a_2-b_2)^2}< 0.
$$
Indeed an ellipse.
\bigskip

On the other hand, looking for line solutions, i.e. functions $y$ in the form
$$
y_l=\alpha x+\beta,\qquad \alpha,\beta\in\mathbb{R},
$$
we have
$$
p(a)=\frac{(x-a_1)y'_x-(y-a_2)}{\sqrt{1+{y'_x}^2}}=\frac{a_1-\alpha a_1-\beta}{\sqrt{1+\alpha^2}}.
$$
And the equation
$$
p(a)p(b)=L^2,
$$
boils down to
$$
(a_1b_1-L^2)\alpha^2+((a_1+b_1)\beta-a_2b_1-a_1b_2)\alpha+(a_2-\beta)(b_2-\beta)-L^2=0,
$$
reducing the number of free parameters to one. Thus $y_l$ is one parameter family of solutions which makes it a \textit{regular} solution.
\end{example}
\begin{example}
The differential equation corresponding to
$$
27a^2p=4p(a)^3,\qquad a\in\mathbb{R},
$$
is as follows
$$
(4a-x)(2x+a)^2 {y'_x}^3-3y(2x+a)(5a-2x){y'_x}^2+(27 a^2+12 y^2(a-x))y'_x-(27a^2-4y^2)y=0.
$$
Differentiation with respect to $y'_x$ yields
$$
3(4a-x)(2x+a)^2 {y'_x}^2-6y(2x+a)(5a-2x){y'_x}+(27 a^2+12 y^2(a-x))=0.
$$
Finally, eliminating $y'_x$ from both equation we obtain
$$
(x^2+y^2)^2-4ax(x^2+y^2)-4a^2y^2=0,
$$
which is an implicit equation of the cardioid.
\end{example}
As we saw with the example of an ellipse, its three representation that we have discussed so far -- i.e. oval, coval and pedal -- are linked together. The following result generalizes this fact.
\begin{proposition}
Consider a curve that has both two-poles oval and coval representations, i.e. it solves
$$
r(a)^2=f(r^2),\qquad \text{and} \qquad p(a)=g(p).
$$ 
Then it also satisfies a pedal equation in the form
$$
f'(r^2)=g'(p).
$$
\end{proposition}
\begin{proof}
Differentiating the oval equation with respect to the arc-length $s$ we obtain
$$
p_c(a)=f'(r^2)p_c,
$$
see the rules for differentiation (\ref{diffrules}).

Similarly, differentiating the coval equation but with respect to $\theta$ yields
$$
p_c(a)=g'(p)p_c.
$$ 
Therefore
$$
f'(r^2)=\frac{p_c(a)}{p_c}=g'(p).
$$
Which is what we want.
\end{proof}
\begin{example}
An ellipse with foci at $0$ and $a$ satisfy:
$$
r+r(a)=M,\quad pp(a)=L^2. 
$$
Thus
$$
f(x)=(M-\sqrt{x})^2,\qquad g(x)=\frac{L^2}{x},
$$
in our case with
$$
f'(x)=-\frac{M-\sqrt{x}}{\sqrt{x}}=1-\frac{M}{\sqrt{x}},\qquad g'(x)=-\frac{L^2}{x^2}.
$$
And therefore the pedal equation
$$
1-\frac{M}{r}=f'(r^2)=g'(p)=-\frac{L^2}{p^2},
$$
or
$$
\frac{L^2}{p^2}=\frac{M}{r}-1,
$$
as claimed.
\end{example}
\begin{example}
For the cardioid we have the coval representation
$$
27a^2p=4p(a)^3,\qquad \Longrightarrow \qquad p(a)=\zav{\frac{27 a^2}{4}p}^{\frac13}=:g(p),
$$
and also, its cartesian equation
$$
(x^2+y^2)^2-4ax(x^2+y^2)-4a^2y^2=0,
$$
is equivalent to 
$$
r(a)^2=2a r+a^2=:f(r^2).
$$
Thus, the resulting pedal equation takes form
$$
\frac{a}{r}=f'(r^2)=g'(p)=\frac{9a^2}{4} \zav{\frac{27 a^2}{4}p}^{-\frac23}.
$$
Simplifying we get the well-known pedal equation
$$
r^3=4a p^2.
$$
\end{example}
\subsection{Principle of triangulation.}
Similarly as with ovals, any coval can be described using only three poles.
\begin{proposition}
For any complex numbers $a,b,c,d\in\mathbb{C}$ it holds
\begin{equation}\label{triangprincgen}
\deter{
\begin{array}{cccc}
p(a) & p(b) & p(c) & p(d)  \\
1 & 1 & 1 & 1  \\
a & b & c & d  \\
\bar a & \bar b & \bar c & \bar d 
\end{array}}=0.
\end{equation}
\end{proposition}
\begin{proof}
Follows directly from the definition of $p(a)$ and properties of determinant.
\end{proof}
If the three points $a,b,c$ are colinear, the equation (\ref{triangprincgen}) becomes trivial; in this case, only two foci are really necessary. 
\begin{proposition}
Assume that the complex numbers $a,b,c\in\mathbb{C}$ are colinear, i.e. it holds
$$
b=\beta a ,\qquad c=\gamma a,
$$
for some real numbers $\beta,\gamma\in\mathbb{R}$. Then we have
\begin{equation}\label{triangprinc}
\deter{
\begin{array}{ccc}
p(a) & p(b) & p(c)   \\
1 & 1 & 1   \\
a & b & c 
\end{array}}=0.
\end{equation}
\end{proposition}
\begin{example}
Principle of triangulation allows us to change the position of foci in the given coval formula. Take, for instance, an ellipse
$$
p(1)p(-1)=L,
$$
with foci located at $-1$ and $1$. How does this formula change if we consider other points? Another natural point for our ellipse is it's center $c:=0$.

Since $1$, $-1$ and $0$ are colinear, we have
$$
\deter{
\begin{array}{ccc}
p(1) & p(-1) & p   \\
1 & 1 & 1   \\
1 & -1 & 0 
\end{array}}=0,\qquad \Longrightarrow \qquad p(-1)=2p-p(1).
$$
Substituting this we obtain a coval representation of an ellipse with it's center in the origin and one focus located at point $1$.
$$
2pp(1)-p(1)^2=L.
$$
\end{example}

\begin{example}
Other interesting points on our ellipse $p(1)p(-1)=L$ are the periapsum $x_0$ and the apoapsum $x_0$, i.e. closest and the farthest points from the focus $-1$ respectively (also known as vertexes). The tanget line at $x_{1}$ is vertical (the same is true for $x_0$). Denoting the distance of this tangent from the origin $x$ we have
$$
p(1)=x-1,\qquad p(-1)=x+1,
$$
thus
$$
p(1)p(-1)=L \qquad \Longrightarrow \qquad x^2-1=L,
$$
and therefore
$$
x_0=-\sqrt{L+1},\qquad x_1=\sqrt{1+L}.
$$
What is the coval equation with these points? Using (\ref{triangprinc}) we have
$$
\deter{
\begin{array}{ccc}
p(1) & p(x_0) & p(x_1)   \\
1 & 1 & 1   \\
1 & x_0 & x_1 
\end{array}}=0,\qquad \Longrightarrow \qquad p(1)=\frac{p(x_0)(x_1-1)-p(x_1)(x_0-1)}{x_1-x_0},
$$
and
$$
\deter{
\begin{array}{ccc}
p(-1) & p(x_0) & p(x_1)   \\
1 & 1 & 1   \\
-1 & x_0 & x_1 
\end{array}}=0,\qquad \Longrightarrow \qquad p(-1)=\frac{p(x_0)(x_1+1)-p(x_1)(x_0+1)}{x_1-x_0}.
$$
Our answer is therefore
$$
\zav{p(x_0)(x_1-1)-p(x_1)(x_0-1)}\zav{p(x_0)(x_1+1)-p(x_1)(x_0+1)}=L(x_1-x_0)^2=4L(1+L),
$$
which can be simplified to
\begin{equation}
\zav{p(x_0)+p(x_1)}^2L+4p(x_0)p(x_1)=4L(1+L).
\end{equation}

\end{example}
\begin{example}
Generally, we can choose any three non-colinear points $a,b,c\in\mathbb{C}$ as poles to describe our ellipse. Using (\ref{triangprincgen}) and some algebra we discover that our original simple formula
$$
p(1)p(-1)=L,
$$ 
transforms into an expression too lengthy to write in full. The condensed version is
$$
a_1 p(a)^2+a_2 p(b)^2+a_3 p(c)^2+a_4 p(a)p(b)+a_5p(a)p(c)+a_6p(b)p(c)=a_7,
$$
where $a_1,\dots, a_7$ are some real numbers (depending on $a,b,c$ and $L$).

This example shows that one must be very careful in choosing the right points for the poles.
\end{example}
\subsection{Poles at infinity}
As we saw, the coval equation for an ellipse is
$$
p(a)p(b)=L,\qquad L>0,
$$
where $a,b$ are its foci.

In the same way we can show that
$$
p(a)p(b)=-L,\qquad L>0,
$$
corresponds to a hyperbola.

What about a parabola? The case $L=0$ does not work. In fact, the singular solution of $p(a)p(b)=0$ is just a line segment connecting $a$ and $b$.

But a parabola can be thought as an ellipse with one of its focus "at infinity", that is we sent, say, $b$ to infinity along some direction while appropriately scaling the curve down so it does not blow up.

Specifically, we will use the following definition.
\begin{definition} Let $z\in \mathbb{C}$. We define the pedal coordinate $p(z\infty)$ with respect to a pole \textit{at infinity} as follows 
\begin{align}
p\zav{z\infty}:=\lim_{s\nearrow \infty}\frac{p\zav{sz}}{s}=p\zav{z}-p.
\end{align}
Furthermore, we will denote the space of complex numbers and points at infinity by
\begin{align}
\mathbb{C}^\star:=\mathbb{C}\cup \szav{z\infty|z\in\mathbb{C}}.
\end{align}
\end{definition}
\begin{remark}
The symbol $p(z\infty)$ is thus just a shorthand for $p(z)-p$ and any coval with poles at infinity is, actually, just an ordinary coval with all its poles finite. 
\end{remark}
Obviously, for $z\not= 0$ it holds
$$
p(z\infty)=|z|p\zav{\frac{z}{|z|}\infty},
$$ 
therefore only the phase of $z$ really matters.

\begin{example}
We now show that the singular solution of
$$
p(a)p\zav{e^{\imag\varphi}\infty}=L,
$$
is, indeed, a parabola.
 
Using the fact that 
$$
p(a)p\zav{e^{\imag\varphi}\infty}=p(a)(p(e^{\imag\varphi})-p),
$$
and using the formula (\ref{difsubs}) we obtain a second order differential equation of the form
$$
f(x,y,{y'}_x)={{y'}_x}^2 P-y'_x Q +R=0,
$$
where
\begin{align*}
P&:=L+\cos\varphi(x-a_1),\\
Q&:=\cos\varphi(y-a_2)+\sin\varphi(x-a_1),\\
R&:=L+\sin\varphi(y-a_2).
\end{align*}
Therefore
$$
\frac{\partial f(x,y,y'_x)}{\partial y'_x}=2y'_x P- Q,
$$
Eliminating $y'_x$ from both equation we obtain an algebraic second order curve
$$
Q^2-4PR=0,
$$
with second degree terms
$$
Q^2-4PR=\cos^2\varphi y^2+2\cos\varphi\sin\varphi xy+\sin^2\varphi x^2+\dots
$$
and, thus, zero discriminant. A parabola.
\end{example}
\begin{example}
Interestingly, an ellipse with both poles at infinity, i.e. the curve
$$
p\zav{e^{\imag\alpha}\infty}p\zav{e^{\imag\beta}\infty}=L,
$$
does not posses any singular solution, but the regular solutions are pairs of lines -- i.e. a degenerate conic.
\end{example}

\section{Coval description of the boundary of a numerical range.}\label{Sec3}
Looking at the definition of $p(a)$ it holds:
\begin{equation}
p(a)=p-\frac{e^{\imag\theta}}{2\imag} \bar a+\frac{e^{-\imag \theta}}{2\imag} a= p+\Im \zav{e^{-\imag \theta}a}=p+\cos\theta\Im(a)-\sin\theta \Re(a).
\end{equation}
We will see that it is natural to consider a matrix analog of this quantity.
\begin{definition}
For any matrix $A\in\mathbb{C}^n$ we define
\begin{equation}\label{pAdef}
p(A):=p I-\frac{e^{\imag\theta}}{2\imag} A^\star+\frac{e^{-\imag \theta}}{2\imag} A= pI+\Im \zav{e^{-\imag \theta}A}=pI+\cos\theta\Im(A)-\sin\theta \Re(A).
\end{equation}
\end{definition}

\begin{proposition}\label{numrangecopolar}
Let $A\in\mathbb{C}^{n\times n}$. Then the curve $\partial W(A)$ satisfies
\begin{equation}\label{copolareq}
\deter{p(A)}=0,
\end{equation}
with the condition that the matrix 
$$
p(A)\geq 0,
$$
is positive semi-definite.
\end{proposition}
\begin{remark}
The condition of positive semi-definiteness implies the following set of $n-1$ inequalities
\begin{equation}\label{copolarineq}
\deter{p(A)_i}\geq 0,\qquad i=1\dots n-1,
\end{equation}
where, for any matrix $X$, $X_i$ denotes the $i$-th principal minor of $X$.

The equation (\ref{copolareq}) together with these inequalities (\ref{copolarineq}) gives what is known as \textit{semidefinite representation} of $\partial W(A)$ (albeit in copolar coordinates). For more info on semidefinite curves see e.g.~\cite{kellipse}.
\end{remark}
\begin{proof}
It can be shown by an optimization argument that the numerical range of a Hermitian matrix is the interval between the smallest and the largest eigenvalue, i.e.
$$
W(A)=[\min \sigma(A),\max \sigma(A)],\qquad A^\star=A.
$$
From this it is clear that
$$
\Re(W(A))=W(\Re(A))=[\min \sigma(\Re A),\max\sigma(\Re A)], \qquad \Re A:=\frac{A+A^\star}{2}
$$
for any matrix $A$. The vertical line $x=\max\sigma(\Re A)$ is actually \textit{tangent} of $\partial W(A)$ with the tangential angle $\theta=\frac{\pi}{2}$. The corresponding pedal coordinate is 
$$
p=\max\sigma(\Re A).
$$
Generally, we can rotate the matrix $A$ by an angle $\phi\in\mathbb{R}$ and compute 
$$
W(\Re e^{\imag\phi}A)=[\min \sigma(\Re e^{\imag\phi} A),\max\sigma(\Re e^{\imag\phi} A)].
$$
A little bit of imagination is needed to understand that the above gives a distance to the tangent $p$ and the tangential angle $\theta$ as follows:
$$
p=\max\sigma(\Re e^{\imag\phi} A),\qquad \theta=\frac{\pi}{2}-\phi.
$$
See Figure~\ref{numconstruction}.
\begin{figure}[h] 
\centering 
\includegraphics[scale=1,trim=1cm 1cm 1cm 0.5cm,clip]{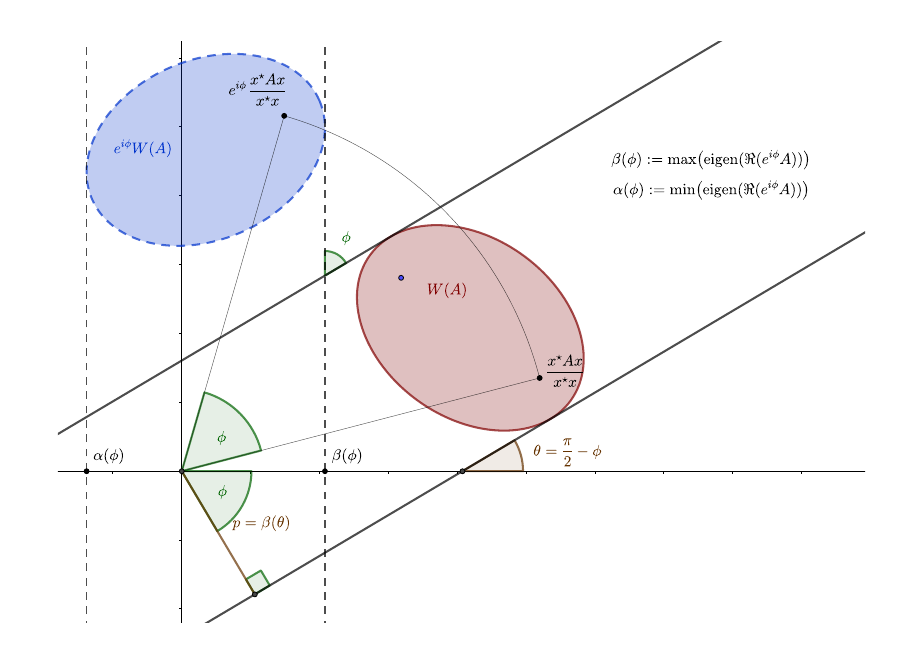}
\caption{Construction of tangents to $\partial W(A)$.}
\label{numconstruction}
\end{figure}
From this it is clear that the co-polar equation of $\partial W(A)$ is
$$
p=\max\sigma(-\Im e^{-\imag\theta} A).
$$
or, using the fact that any eigenvalue solves the characteristic equation, 
$$
|pI+\Im e^{-\imag\theta} A|=0,
$$
which proves (\ref{copolareq}).

Since $p$ is the \textit{maximal} eigenvalue of $-\Im e^{-\imag\theta} A$, the matrix
$$
p(A)=pI+\Im e^{-\imag\theta} A,
$$
is necessarily positive semi-definite. 
\end{proof}

\begin{remark}
The equation (\ref{copolareq}) is actually equivalent to the well known result due to Kipenhahn \cite{Kippenhahnorig,Kippenhahn} that the dual curve of $\partial W(A)$ is given by Kippenhahn polynomial, i.e. algebraic curve given implicitly by
$$
\deter{I-x\Re A-y\Im A}=0
$$
For the sake of brevity, we will not give any details here.
\end{remark}
We are going to show the following elegant proof of Proposition \ref{n2prop}, i.e. that the boundary of the numerical range of a 2 by 2 matrix is an ellipse.
\begin{proof}
The Shur's theorem says that every square matrix is unitary similar to an upper triangular matrix. Since the numerical range is invariant to unitary transformation,  we can work only with upper triangular matrices. Let us, therefore, assume that 
$$
A=\zav{\nadsebou{\lambda_1}{0}\nadsebou{b}{\lambda_2}}.
$$  
It is easy to see that
$$
p(A)=\zav{\nadsebou{p(\lambda_1)}{-\alpha \bar b}\nadsebou{-\bar \alpha b}{p(\lambda_2)}},\qquad \alpha:=\frac{e^{\imag \theta}}{2\imag}.
$$
Thus
$$
\deter{p(A)}=p(\lambda_1)p(\lambda_2)-\alpha\bar\alpha b\bar b=p(\lambda_1)p(\lambda_2)-\frac{|b|^2}{4}.
$$
Therefore the curve $\partial W(A)$ is described by
$$
p(\lambda_1)p(\lambda_2)=\frac{|b|^2}{4},
$$
which is an ellipse with foci at $\lambda_1$, $\lambda_2$. Furthermore the number 
$$
|b|=\sqrt{\tr (A^\star A) -|\lambda_1|^2-|\lambda_2|^2};
$$ is the minor axis.
\end{proof}
For dimensions greater than 2 we are going to need a more effective method of computation. We are going to present a kind of ``greedy algorithm''.
But first, let us prove some simple properties of the matrix $p(A)$.
\begin{proposition}
Let $A\in\mathbb{C}^{n\times n}$ with (not necessarily distinct) eigenvalues $\lambda_1,\dots \lambda_n$. Then the following holds:
\begin{align}
\frac{u^\star p(A) u}{u^\star u}&=p\zav{\frac{u^\star A u}{u^\star u}}.\\
W(p(A))&=p(W(A)):=\szav{p(z); z\in W(A)}.\\
\tr \zav{p(A)}&=n p\zav{\frac{\tr A}{n}}.
\end{align}
\end{proposition}
\begin{proof}
The first formula is due to the fact that 
$$
\frac{u^\star p(A) u}{u^\star u}=\frac{u^\star \zav{pI-\bar\alpha A-\alpha A^\star} u}{u^\star u}=p-\bar\alpha \frac{u^\star A u}{u^\star u}-\alpha \frac{u^\star A^\star u}{u^\star u}=p\zav{\frac{u^\star A u}{u^\star u}}.
$$
Here
$$
\alpha:=\frac{e^{2\imag \theta}}{2\imag}.
$$
The rest is an easy exercise.
\end{proof}

\subsection{Greedy algorithm}
\begin{lemma}\label{greedyalglemma}
Let $f(x,y,z)$ be a complex polynomial in real variables $x$, $y$, $z$ with the following properties:
\begin{align*}
1)& & \forall t\in\mathbb{R}:\qquad f(tx,ty,tz)&=t^{m}f(x,y,z), & \text{homogeneous of degree $m$,}\\
2)& & \overline{f(x,y,z)}&=f(x,z,y), & \text{symmetric in the last variables.}
\end{align*}
Define a function $\tilde f(x,y,z)$ by the following rules:
\begin{align}
	\intertext{ If $\forall x:$ $f(x,0,1)= 0$ then} 
		\tilde f(x,y,z)&:=\frac{f(x,y,z)}{yz}.\\
	\intertext{ If $f(x,0,1)=\gamma \prod_{j=1}^m(x- \lambda_j) $, $\gamma\not=0$ then}
	\label{Lemmamcase}
	\tilde f(x,y,z)&:=\frac{f(x,y,z)-f(1,0,0)\prod\limits_{j=1}^m(x-\bar\lambda_j y-\lambda_j z)}{yz}.\\
	\intertext{ If $f(x,0,1)=\gamma\prod_{j=1}^k(x-\lambda_j)$, $\gamma\not=0$, $0\leq k<m$, then}
	\label{Lemmakcase}
	\tilde f(x,y,z)&:=\frac{f(x,y,z)-\prod\limits_{j=1}^k(x-\bar\lambda_j y-\lambda_j z)\prod\limits_{\beta^{m-k}=-\gamma}(-\bar\beta y-\beta z)}{yz}.	
	\end{align}
Then the function $\tilde f$ is a homogeneous polynomial of degree $m-2$ which is symmetric in the  last variables. In other words, it holds
\begin{align*}
 \forall t\in\mathbb{R}:\qquad \tilde f(tx,ty,tz)&=t^{m-2}\tilde f(x,y,z), &
 \overline{\tilde f(x,y,z)}&=\tilde f(x,z,y). & 
\end{align*}
\end{lemma}
\begin{proof} 
Notice that by homogeneity $f(x,0,1)$ is a polynomial in $x$ of degree at most $m$. Thee proof is split into three cases:
\bigskip

\textbf{Case 1).} The polynomial $f(x,0,1)$ is identically zero (in other words, it has degree $-\infty$). Therefore the polynomial $f(x,y,z)$ as a function of $y$ has a zero in $y=0$ since $f(x,0,z)=z^m f\zav{x/z,0,1}=0$ (and, by continuity, $f(x,0,0)=0$). Therefore $f$ can be written as $f(x,y,z)=y f_2(x,y,z)$. Because of the symmetry, there is also the factor of $z$: $f(x,y,z)=\overline{f(x,z,y)}=z\overline{f_2(x,z,y)}$. Thus the function
$$
\tilde f(x,y,z):=\frac{f(x,y,z)}{yz},
$$  
is, indeed, a polynomial. The homogeneity and the symmetry are obvious.
\bigskip

\textbf{Case 2).} The polynomial $f(x,0,1)$ is of degree $m$. Therefore by Fundamental theorem of Algebra there exists complex numbers $\lambda_j$ (not necessarily distinct) and a complex number $\gamma\not=0$ such that $f(x,0,1)=\gamma \prod_{j=1}^m(x- \lambda_j) $. 
 
Note that we have $\gamma=f(1,0,0)$ since 
$$
f(x,y,z)=f(x,0,0)+O(y,z)=x^mf(1,0,0)+O(y,z).
$$ 
Also $\gamma$ is a real number, since $\gamma=f(1,0,0)=\overline{f(1,0,0)}=\bar \gamma$ by symmetry.
We now show that 
$$
f(x,0,z)=f(1,0,0)\prod_{j=1}^m\zav{x-\lambda_j z},\qquad
f(x,y,0)=f(1,0,0)\prod_{j=1}^m\zav{x-\bar\lambda_j y}.
$$
This can be seen purely algebraically:
\begin{align*}
f(x,0,z)&=z^m f\zav{\frac{x}{z},0,1}=z^m f(1,0,0)\prod_{j=1}^m\zav{\frac{x}{z}-\lambda_j}=f(1,0,0)\prod_{j=1}^m\zav{x-\lambda_j z}.\\
\intertext{And}
f(x,y,0)&=\overline{f(x,0,y)}=\overline{y^mf\zav{\frac{x}{y},0,1}}= \overline{y^m f(1,0,0)\prod_{j=1}^m\zav{\frac{x}{y}-\lambda_j}}\\
&=
\overline{f(1,0,0)}\prod_{j=1}^m \zav{x-\bar \lambda_j y}=f(1,0,0)\prod_{j=1}^m \zav{x-\bar \lambda_j y}.
\end{align*}
Hence the numerator of $\tilde f$ in (\ref{Lemmamcase}) has zero both in $y=0$ and $z=0$ and thus $\tilde f$ itself is, indeed, a polynomial. 

The fact that $\tilde f$ is homogeneous of degree $m-2$ as well as symmetric, is obvious.
\bigskip

\textbf{Case 3).} The polynomial $f(x,0,1)$ has degree $0\leq k<m$ therefore it is equal to $f(x,0,1)=\gamma\prod_{j=1}^k(x-\lambda_j)$ for some complex numbers $\gamma,\lambda_j$. We are going to proceed by a limit argument. We introduce the following polynomial
$$
f_s(x,y,z):=sf(x,y,z)+(1-s)x^{m-k}\prod\limits_{j=1}^k\zav{x-\bar\lambda_j y-\lambda_j z},\qquad s\in [0,1],
$$
which interpolates between $f_1(x,y,z)=f(x,y,z)$ and $f_0(x,y,z)=x^{m-k}\prod_{j=1}^k\zav{x-\bar\lambda_j y-\lambda_j z}$. Note that $f_s$ is homogeneous of degree $m$ and symmetric in the last variables. Also, crucially, 
$$
f_s(1,0,0)=1-s\not=0,\qquad  s<1, 
$$ 
and there are $m$ solutions of 
$$
f_s(x,0,1)=sf(x,0,1)+(1-s)x^{m-k}\prod_{j=1}^k(x-\lambda_j)=\prod_{j=1}^k(x-\lambda_j)\zav{s\gamma+(1-s)x^{m-k}}=0,
$$ 
specifically $\lambda_j$ for $j=1\dots k$ and 
$$
\lambda_{k+l+1}=\beta_l \zav{\frac{s}{1-s}}^{\frac{1}{m-k}},\qquad \beta_l:=(-\gamma)^{\frac{1}{m-k}}e^{\frac{2\pi\imag l}{m-k}},\qquad l=0\dots m-k-1.
$$
Therefore we can apply the previous case to obtain
\begin{align*}
\tilde f_s(x,y,z)&=\frac{f_s(x,y,z)-(1-s)\prod\limits_{j=1}^{m}(x-\bar \lambda_j y-\lambda_j z)}{yz}\\
&=\frac{f_s(x,y,z)-(1-s)\prod\limits_{j=1}^k(x-\bar\lambda_j y-\lambda_j z)\prod\limits_{l=0}^{m-k-1}\zav{x-\bar \beta_l \zav{\frac{s}{1-s}}^{\frac{1}{m-k}}y-\beta_l \zav{\frac{s}{1-s}}^{\frac{1}{m-k}}z}}{yz},
\end{align*}
is a $m-2$ degree homogeneous polynomial which is symmetric in $y,z$.
Taking the limit $s\nearrow 1$ yields (\ref{Lemmakcase}).

\end{proof}
Let us introduce a notation that simplifies things a bit for the following discussion.
\begin{definition}
Let $A\in\mathbb{C}^{n\times n}$. We will denote by $\bar A$ \textit{any} member of $\mathbb{C}^{n\times n}$ such that the following properties holds
\begin{align}
\bar AA&=A\bar A,\\
\overline{Q^{-1}AQ}&=Q^{-1}\bar A Q,\qquad \forall Q: \deter{Q}\not=0,\\
\text{If }Au&=\lambda u\qquad \text{  then  }\qquad\bar A u=\bar \lambda u.
\end{align}
\end{definition}
%
\begin{remark}  
For matrices $A$ with distinct eigenvalues the matrix $\bar A$ is unique and it corresponds to a standard matrix function $f(A)$ for $f(z):=\bar z$. In the degenerate case the matrix $\bar A$ is not unique but it does not matter for our purposes.  We are introducing this notation just for the sake of brevity, since we can write, e. g.
$$
\tr \bar A=\sum_{j=1}^n \overline{\lambda_j}=\overline{\tr A},\qquad \tr \zav{A\bar A}=\sum_{j=1}^n\abs{\lambda_j}^2, 
$$
or
$$
\tr \zav{{\rm adj }(xI-A)\bar A}=\deter{xI-A}\sum_{j=1}^n\frac{\overline{\lambda_j}}{x-\lambda_j},\qquad \forall x\not\in\sigma(A). 
$$
\end{remark}
\begin{definition}
For $A\in\mathbb{C}^{n\times n}$ let
$$
t_k(A):=\tr A^k\zav{A^\star -\bar A}=\tr A^k A^\star-\sum_{j=1}^n \lambda_j^k\bar \lambda_j,
$$
where $\lambda_j$-s are the eigenvalues of $A$.
\end{definition}
\begin{remark}
The quantity $t_1(A)\geq 0$ is just the sum of absolute values of all the off-diagonal entries of the Schur form of $A$ and it measures how much the matrix $A$ deviates from being normal, since if $A^\star A=AA^\star$ we have $t_1(A)=0$. 
\end{remark}
We are going to prove Theorem~\ref{maintint}. Let us restate it with more details.
\begin{theorem}\label{maint}
Let $A\in\mathbb{C}^{n\times n}$, $n\geq 2$. Then we have
\begin{equation}\label{covalpa}
\deter{p(A)}=\prod_{j=1}^{n}p(\lambda_j)-\frac{t_1(A)}{4}\prod_{j=1}^{n-2}p\zav{\lambda^{(2)}_j}+\sum_{k=2}^{\trunc{\frac{n}{2}}}\frac{\gamma_k}{4^k}\prod_{j=1}^{n-2k}p\zav{\lambda^{(k+1)}_{j}}.
\end{equation}
Here, $\gamma_k\in \mathbb{R}$,  $\lambda_j$ are the eigenvalues of $A$ and $\lambda^{(2)}_j$ are solutions of 
\begin{equation}\label{anticharacteristiceq}
\tr {\rm adj}\zav{x I-A}\zav{A^\star-\bar A}=0.
\end{equation}
The remaining poles might be at infinity, i.e. $\lambda_{j}^{(k)}\in \mathbb{C}^\star$ for $k>2$ in general. 
\end{theorem}
\begin{remark}
In the equation (\ref{covalpa}) we are adhering to the standard custom that ``empty'' sums and products equals to their respective neutral elements:
$$
\forall f:\qquad \sum_{k=2}^{1}f(k):=0,\qquad \prod_{j=1}^0f(j):=1.
$$
\end{remark}
\begin{proof}
Let 
$$
f(x,y,z):=\deter{xI-yA^\star -z A}.
$$
The function $f$ obviously satisfy hypothesis of Lemma~\ref{greedyalglemma}, that is $f$ is a homogeneous polynomial of degree $n$, and symmetric in the last variables. The leading coefficient is $f(1,0,0)=1$ and the roots of
$$
f(x,0,1)=\deter{xI-A},
$$
are the eigenvalues of $A$ denoted $\lambda_j$. Therefore we have by Lemma~\ref{greedyalglemma}
$$
\tilde f(x,y,z)=\frac{f(x,y,z)-\prod\limits_{j=1}^n (x-\bar \lambda y -\lambda z)}{yz}=\frac{\deter{xI-yA^\star -z A}-\deter{xI-y\bar A -z A}}{yz}.
$$
Or, in another words,
$$
\deter{xI-yA^\star -z A}=\deter{xI-y\bar A -z A}+yz \tilde f(x,y,z).
$$
We are going to repeat the process on $\tilde f$. Using the well known formula
$$
\frac{\partial\deter{X}}{\partial y}=\tr \zav{{\rm adj}(X) \frac{\partial X}{\partial y}}.
$$
we can see that 
$$
\tilde f(x,0,1)=-\tr {\rm adj}(xI-A)(A^\star-\bar A),
$$
The $n-2$ roots of this polynomial are denoted $\lambda_j^{(2)}$.
Later we will see an effective method how to verify that the leading coefficient is indeed
$$
\tilde f(1,0,0)=-t_1(A).
$$ 
This has an interesting consequence in the sense that when the quantity $t_1(A)$ is zero, the matrix is normal, thus $\bar A=A^\star$ and $\tilde f(x,y,z)\equiv 0$ identically. Therefore we either have all $n-2$ roots $\lambda^{(2)}_j$ or the process terminates after first step. Nothing in between. The third case of Lemma~\ref{greedyalglemma} hence does not apply and we can write generally
$$
\tilde{\tilde f}(x,y,z)=\frac{\tilde f(x,y,z)+t_1(A)\prod\limits_{j=1}^{n-2}(x-\overline{\lambda^{(2)}_j}y-\lambda_j^{(2)}z)}{yz} 
$$ 
or
\begin{equation}\label{detergreedymain}
\deter{xI-yA^\star -z A}=\deter{xI-y\bar A -z A}-yz t_1(A) \prod\limits_{j=1}^{n-2}(x-\overline{\lambda^{(2)}_j}y-\lambda_j^{(2)}z)+(yz)^2 \tilde{\tilde f}(x,y,z).
\end{equation}
Following iterations are much less clear, but successive application of Lemma~\ref{greedyalglemma} guaranties that there exists expansion of the form:
\begin{equation}\label{detergreedy}
(yz)^2\tilde{\tilde f}(x,y,z)=\sum_{k=2}^{\trunc{\frac{n}{2}}}\gamma_k (yz)^k\prod_{j=1}^{l_k}\zav{x-y\overline{\lambda^{(k+1)}_{j}}-z\lambda^{(k+1)}_{j}}\hspace{-5mm}\prod_{\beta^{n-2k-l_k}=\delta_k}\hspace{-5mm}\zav{-\bar \beta y -\beta z}.
\end{equation}
for some complex numbers $\lambda^{(k+1)}_j, \delta_k$.

The proof now follows simply substituting $x=p$, $ y=\frac{e^{\imag \theta}}{2\imag}$, $z=-\frac{e^{-\imag \theta}}{2\imag}$ into (\ref{detergreedymain},\ref{detergreedy}).

\end{proof}
\begin{remark}
The proof actually tells us more about the potential poles at infinity.
Say for some matrix $A$ we have that
$
\lambda_1^{(k)},\dots\lambda_m^{(k)} 
$
are at infinity for a fixed $k$,i.e.
$$
\lambda_j^{(k)}=\omega_j\infty,\qquad j=1\dots m.
$$
Then it follows that there exist a complex number $q\in\mathbb{C}$ such that
$$
\omega_j^m=q,\qquad \forall j=1,\dots,m.
$$
\end{remark}
\begin{example}
For a normal matrix $A$ we have 
$$
\deter{p(A)}=\prod_{j=1}^{n}p(\lambda_j),
$$
since the Schur form of $A$ is diagonal. 
Any curve that solves 
\begin{equation}\label{numrangenormal}
\prod_{j=1}^{n}p(\lambda_j)=0
\end{equation}
must satisfy 
$$
p(\lambda_k)=0,
$$
for some $k$ solutions of which is either a line passing through $\lambda_k$ or the point $\lambda_k$ itself.
The general solutions of (\ref{numrangenormal}) is thus any curve that is piece-wise a segment of a line passing through some $\lambda_k$ or these points where it can turn sharply.

The boundary of the convex hull of the eigenvalues is obviously one such curve. Furthermore, it is the only solution that encloses a convex set containing all the eigenvalues. 

This therefore reproduces the well known result that the numerical range of a normal matrix is the convex hull of its eigenvalues.
\end{example}

It is possible to get an explicit form for the equation (\ref{anticharacteristiceq}) in terms of $t_j(A)$ and eigenvalues of $A$.
\begin{lemma}
Let $A\in\mathbb{C}^{n\times n}$ with eigenvalues $\lambda_1,\dots, \lambda_n$. It holds
\begin{equation}\label{anticharexplicit}
\tr {\rm adj}(xI-A)(A^\star -\bar A)=\sum_{j=1}^{n-1}\sum_{k=0}^{n-1-j}t_j(A)x^{n-1-j-k}(-1)^k e_k(A),
\end{equation}
where $e_k(A)$ are elementary symmetric polynomials
$$
e_k(A):=\sum_{1\leq j_1\leq j_2\leq\dots \leq j_k\leq n }\lambda_{j_1}\cdots \lambda_{j_n}.
$$
\end{lemma}
\begin{remark}
In particular we have
\begin{align}
n&=2 & \tr {\rm adj}(x I-A)(A^\star-\bar A)&=t_1(A).\\
n&=3 & \tr {\rm adj}(x I-A)(A^\star-\bar A)&=t_1(A)(x-e_1(A))+t_2(A).\\
n&=4 & \tr {\rm adj}(x I-A)(A^\star-\bar A)&=t_1(A)(x^2-e_1(A)x+e_2(A))+t_2(A)(x-e_1(A))+t_3(A).
\end{align}

\end{remark}
\begin{proof}
Remember, the generating function for elementary symmetric polynomials $e_k(A)$ is 
$$
\deter{xI-A}=\sum_{k=0}^n x^{n-k}(-1)^ke_k(A).
$$ 
since any matrix is a solution of its characteristic equation (Cayley–Hamilton theorem) we have
$$
\sum_{k=0}^n A^{n-k}(-1)^ke_k(A)=0.
$$
We claim that
\begin{equation}\label{adjformula}
{\rm adj}(xI-A)=\sum_{j=0}^{n-1}\sum_{k=0}^{n-1-j} A^j x^{n-1-j-k}(-1)^k e_k(A).
\end{equation}
To prove this, observe
\begin{align*}
x{\rm adj}(xI-A)&=\sum_{j=0}^{n-1}\sum_{k=0}^{n-1-j} A^j x^{n-j-k}(-1)^k e_k(A),\\
\intertext{shifting the summation index $j\to j+1$ we get}
&=\sum_{j=-1}^{n-2}\sum_{k=0}^{n-2-j} A^{j+1} x^{n-1-j-k}(-1)^k e_k(A),\\
&=\sum_{j=-1}^{n-2}\sum_{k=0}^{n-1-j} A^{j+1} x^{n-1-j-k}(-1)^k e_k(A)-\sum_{j=-1}^{n-2}A^{j+1}(-1)^{n-1-j}e_{n-1-j}(A).
\end{align*}
\begin{align*}
\intertext{In the second term let change the summation index as follows $j=n-1-k$:}
\sum_{j=-1}^{n-2}A^{j+1}(-1)^{n-1-j}e_{n-1-j}(A)&=\sum_{k=1}^{n}A^{n-k}(-1)^{k}e_{k}(A)=-A^n,\\
\intertext{due to Caley-Hamilton theorem. In the first term we split the summation into three parts: }
\sum_{j=-1}^{n-2}\sum_{k=0}^{n-1-j} A^{j+1} x^{n-1-j-k}(-1)^k e_k(A)&=
\zav{\sum_{j=0}^{n-1}\sum_{k=0}^{n-1-j}+\delta_{j,-1}\sum_{k=0}^{n}-\delta_{j,n-1}\delta_{k,0}}A^{j+1} x^{n-1-j-k}(-1)^k e_k(A)\\
&=A{\rm adj}(xI-A)+\sum_{k=0}^n x^{n-k}(-1)^ke_k(A)I-A^n\\
&=A{\rm adj}(xI-A)+\deter{x-A}I-A^n.
\end{align*}
Substituting both these result into the above we obtain
$$
x{\rm adj}(xI-A)=A{\rm adj}(xI-A)+\deter{xI-A}I,
$$
or
$$
(xI-A){\rm adj}(xI-A)=\deter{xI-A}I,
$$
the defining property of the matrix ${\rm adj}(xI-A)$.
The formula (\ref{anticharexplicit}) now follows by summing over $j$ in (\ref{adjformula}).
\end{proof}
\subsection{Numerical range of a 3 by 3 matrix}
\begin{corollary}\label{3by3cor}
Let $A\in\mathbb{C}^{3\times 3}$ with $\lambda_1,\lambda_2,\lambda_3$ as its eigenvalues. Then the curve $\deter{p(A)}=0$ satisfies
\begin{equation}\label{covalpa3}
p(\lambda_1)p(\lambda_2)p(\lambda_3)=\frac{t_1(A)}{4}p(\chi),
\end{equation}
where the number $\chi$ is given by 
\begin{equation}\label{lambda2formula}
\chi:=\tr A-\frac{t_2(A)}{t_1(A)},\qquad t_1(A)\not=0.
\end{equation}
Also
\begin{equation}\label{chiinWA}
\chi\in W(A).
\end{equation}
\end{corollary}

\begin{proof}
The form of (\ref{covalpa3}) is the direct consequence of Theorem~\ref{maint}. The formula (\ref{lambda2formula}) follows from (\ref{anticharexplicit}).

For the second fact (\ref{chiinWA}), suppose we have a Schur form of $A$:
$$
A:=\zav{\begin{array}{ccc}
\lambda_1 & b_1 & b_2 \\
0 & \lambda_2 & c_1 \\
0 & 0 & \lambda_3 
\end{array}
},
$$
Then
$$
A^\star=\zav{\begin{array}{ccc}
\bar \lambda_1 & 0 & 0 \\
\bar b_1 & \bar \lambda_2 & 0 \\
\bar b_2 & \bar c_1 & \bar \lambda_3 
\end{array}
},
\qquad
\bar A=\zav{\begin{array}{ccc}
\bar \lambda_1 & * & * \\
0 & \bar \lambda_2 & * \\
0 & 0 & \bar \lambda_3 
\end{array}
},
$$
where stars $*$ represents some numbers that are irrelevant for the subsequent discussion.
Thus
$$
A^\star-\bar A =\zav{\begin{array}{ccc}
0 & * & * \\
\bar b_1 & 0 & * \\
\bar b_2 & \bar c_1 & 0 
\end{array}
},
$$
and
$$
t_1(A)=\tr A(A^\star-\bar A)=\tr \zav{\begin{array}{ccc}
\lambda_1 & b_1 & b_2 \\
0 & \lambda_2 & c_1 \\
0 & 0 & \lambda_3 
\end{array}
}\zav{\begin{array}{ccc}
0 & * & * \\
\bar b_1 & 0 & * \\
\bar b_2 & \bar c_1 & 0 
\end{array}
}=|b_1|^2+|b_2|^2+|c_1|^2.
$$
Also
\begin{align*}
t_2(A)&=\tr A^2(A^\star-\bar A)=\tr \zav{\begin{array}{ccc}
\lambda_1^2 & b_1(\lambda_1+\lambda_2) & b_2(\lambda_1+\lambda_3)+b_1c_1 \\
0 & \lambda_2^2 & c_1(\lambda_2+\lambda_3) \\
0 & 0 & \lambda_3^2 
\end{array}
}\zav{\begin{array}{ccc}
0 & * & * \\
\bar b_1 & 0 & * \\
\bar b_2 & \bar c_1 & 0 
\end{array}
}\\
&=|b_1|^2(\lambda_1+\lambda_2)+|b_2|^2(\lambda_1+\lambda_3)+|c_1|^2(\lambda_2+\lambda_3)+\bar b_2 b_1 c_1.
\end{align*}
Substituting these values for $t_1(A), t_2(A)$ into (\ref{lambda2formula}) yields
\begin{equation}\label{chiexplicit}
\chi=\frac{u A u^\star}{u u^\star},\qquad u:=(c_1,-b_2,b_1).
\end{equation}
\end{proof}
\begin{example}
Perhaps the best feature of Theorem~\ref{maint} is that it allows us to find matrices with prescribed Kippenhahn curve -- \textit{if} a coval representation of the Kippenhahn curve is known.

Take, for instance, the cardioid (\ref{covalcardioid}):
$$
4p(a)^3=27a^2p,\qquad a>0.
$$
and say we want to find matrix $A$ such that
$$
\deter{p(A)}=p(a)^3-\frac{27}{4}a^2 p.
$$
It is evident from Corollary~\ref{3by3cor} that $A$ has to have an eigenvalue $a$ (with multiplicity 3), and that
$$
t_1(A)=27a^2,\qquad \chi=0.
$$
Therefore the Schur form of $A$ must be
$$
A=\zav{\begin{array}{ccc}
a & b_1 & b_2 \\
0 & a & c_1 \\
0 & 0 & a 
\end{array}
},
$$
and the numbers $b_1,b_2,c_1$ must satisfy
$$
|b_1|^2+|b_2|^2+|c_1|^2=27a^2,\qquad \bar b_2b_1c_1=27a^3.
$$
 These are 3 equation (the third is the conjugation of the second equation) on 6 real unknown. Or -- if we restrict ourselves to real matrices, 2 equations for three unknowns.

One particularly nice solution is 
$$
b_1=b_2=c_1=3a,
$$
and the resulting matrix is in the form
$$
A=a\zav{\begin{array}{ccc}
1 & 3 & 3 \\
0 & 1 & 3 \\
0 & 0 & 1 
\end{array}
}.
$$
In a later publication we are going to see that the same line of reasoning can be applied to many more curves, particularly to the family of so-called \textit{sinusoidal spirals}, which the cardioid is a member of.
\end{example}

Two facts are immediately obvious from the formula (\ref{covalpa3}).
\begin{corollary}
If the number $\chi$ coincides with some eigenvalue, say $\chi=\lambda_1$
the curve $\deter{p(A)}=0$ reduces to an ellipse an a point. Specifically
$$
\deter{p(A)}=p(\lambda_1)\zav{p(\lambda_2)p(\lambda_3)-\frac{t_1(A)}{4}}.
$$
\end{corollary}
\begin{remark}
This result is not new -- see \cite[Theorem 2.2]{RodmanSpitkovsky}. But from the coval representation (\ref{covalpa3}) it is obvious.
\end{remark}
\begin{remark}
The corresponding numerical range boundary $\partial W(A)$ is therefore, in this case, either ellipse, i.e. the singular solution of
$$
p(\lambda_2)p(\lambda_3)-\frac{t_1(A)}{4},
$$
if $\chi$ lies inside it or a comet of the same ellipse with its tail consisting of two tangent line segments connecting the point $\chi$ and the ellipse. 
\end{remark}

\begin{remark}
The point $\chi$ thus effectively behaves like sort of  ``anti-eigenvalue'' in the sense that when it coincides with one of the eigenvalues the resulting numerical range behaves as if the eigenvalue was never there (ignoring the possible tail segments). In other words, $\chi$ ``annihilates'' eigenvalues.

(Since the term ``anti-eigenvalue'' is already taken, see \cite{Gustafson1}, we will refer to the number $\chi$ by other name -- ``the secondary value of $A$''. )

We will see that this behavior of $\chi$ is (partially) maintained even for all dimensions of the matrix $A$.  
\end{remark}
The second observation we can made is the following.
\begin{corollary}\label{innerlines3}
Any line connecting $\chi$ with some eigenvalue $\lambda_j$, $j=1,2,3$ is a tangent of $\deter{p(A)}=0$. Moreover, these three lines are the only tangents of $\deter{p(A)}=0$ that passes through $\chi$.
\end{corollary}
\begin{proof}
Consider a general point $a\in\mathbb{C}$. Any line that passes through $a$ satisfies $p(a)=0$. So for a line going through $\chi$ we have $p(\chi)=0$ and substituting this into (\ref{covalpa3}) we get
$$
p(\lambda_1)p(\lambda_2)p(\lambda_3)=\frac{t_1(A)}{4}p(\chi)=0.
$$
Thus a necessary condition for this line to be a solution of (\ref{covalpa3}) (and thus a tangent) is $p(\lambda_j)=0$ for some $j$. See Figure~\ref{threelines2component}.
\end{proof}
\begin{figure}[ht] 
\centering 
\includegraphics[scale=0.40,trim=2cm 9cm 2cm 2cm,clip]{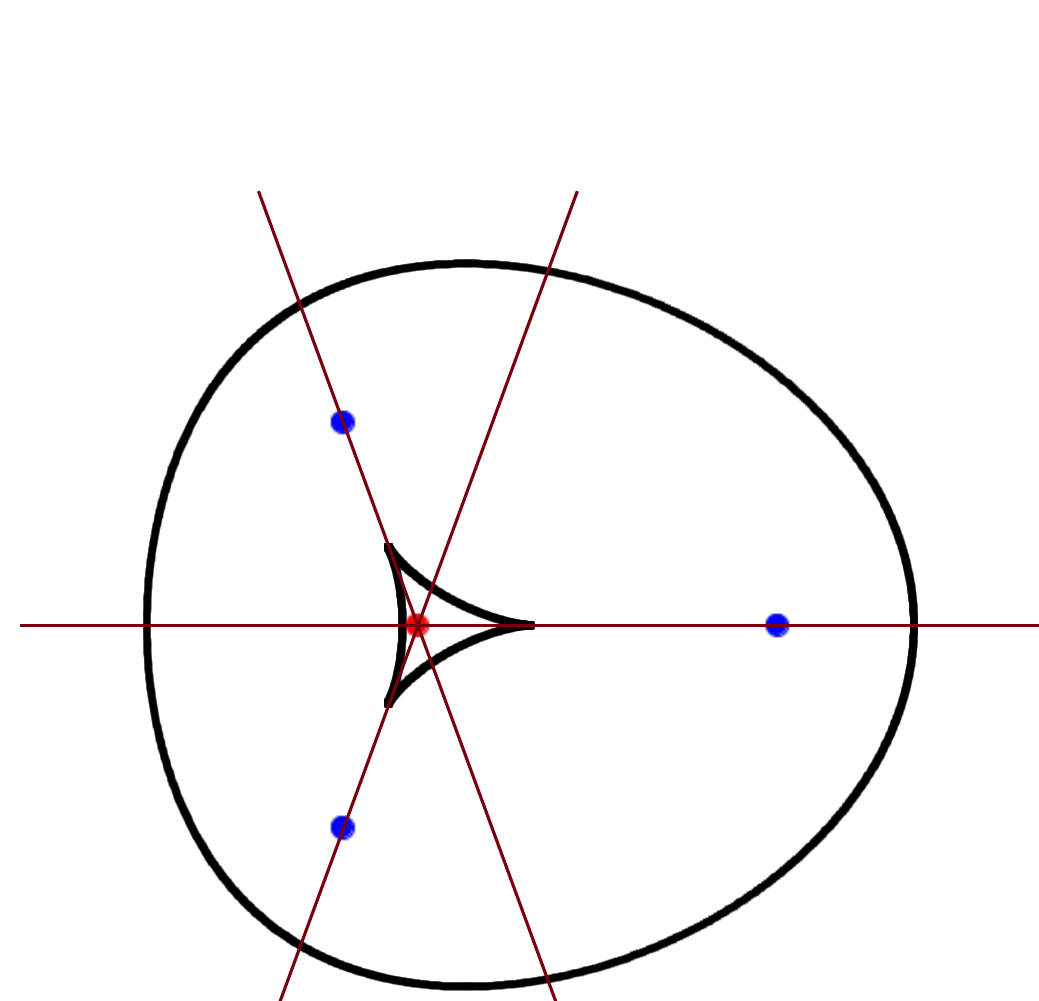}
\includegraphics[scale=0.40,trim=2cm 9cm 2cm 2cm,clip]{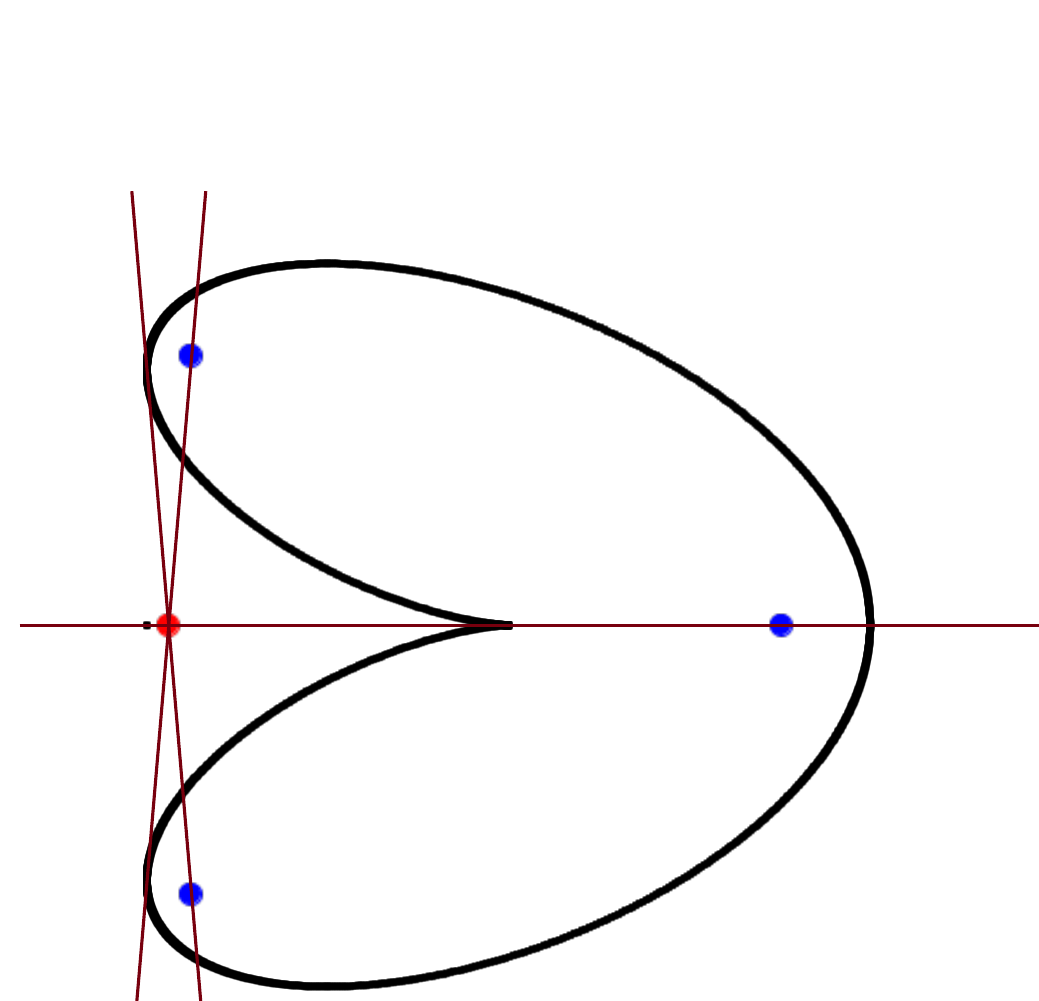}
\caption{Kippenhahn curves of a 3 by 3 matrix and lines connecting the eigenvalues with the ``anti eigenvalue''. On the left two component case, on the right a single component.}
\label{threelines2component}
\end{figure}

\subsection{Numerical range of a 4 by 4 matrix}
\begin{corollary}\label{4by4case}
Let $A\in\mathbb{C}^{4\times 4}$ with $\lambda_1,\lambda_2,\lambda_3, \lambda_4$ as its eigenvalues. Then the curve $\deter{p(A)}=0$ satisfies
\begin{equation}\label{covalpa4}
p(\lambda_1)p(\lambda_2)p(\lambda_3)p(\lambda_4)=\frac{t_1(A)}{4}p(\chi_+)p(\chi_-)-\frac{\gamma_2}{16},
\end{equation}
where $\chi_{\pm}$ are solutions of 
\begin{equation}\label{lambda2jformula}
t_1(A) x^2- (t_1(A)\tr(A)-t_2(A))x-\deter{A}t_{-1}(A)=0,
\end{equation}
and 
\begin{equation}
\gamma_2:=\deter{A}\zav{e_2\zav{A^{-1}A^\star}-e_2\zav{A^{-1}\bar A }}+t_1(A)\zav{\chi_+\overline{\chi_-}+\overline{\chi_+}\chi_-}.
\end{equation}
Also
\begin{equation}\label{centroid4by4}
\chi:=\frac{\chi_++\chi_-}{2}\in W(A).
\end{equation}
\end{corollary}
\begin{proof}
The only nontrivial fact here is (\ref{centroid4by4}). From (\ref{lambda2jformula}) we have
$$
\chi=\frac12\zav{\tr(A)-\frac{t_2(A)}{t_1(A)}}.
$$
Without loss of generality we may assume that $A$ is upper triangular
$$
A=\zav{\begin{array}{cccc}
\lambda_1 & b_1 & b_2 & b_3 \\
0 & \lambda_2 & c_1 & c_2 \\
0 & 0 & \lambda_3 & d_1 \\
0 & 0 & 0 & \lambda_4
\end{array}}.
$$ 
Define the following column vectors $u_1, u_2, u_3, u_4$
$$
u_1:=\zav{\begin{array}{c} 0 \\ d_1 \\ -c_2 \\ c_1 \end{array}},\qquad
u_2:=\zav{\begin{array}{c} d_1 \\ 0 \\ -b_3 \\ b_2 \end{array}},\qquad
u_3:=\zav{\begin{array}{c} c_2 \\ -b_3 \\ 0 \\ b_1 \end{array}},\qquad
u_4:=\zav{\begin{array}{c} c_1 \\ -b_2 \\ b_1 \\ 0 \end{array}}.
$$
%
Simple calculation shows that 
$$
\sum_{j=1}^4 u_j^\star u_j=2t_1(A)
$$
and
$$
\sum_{j=1}^4 u_j^\star A u_j= \tr(A)t_1(A)-t_2(A)=2t_1(A)\chi.
$$
Thus 
$$
\chi= \sum_{j=1}^4 \frac{u_j^\star A u_j}{u_j^\star u_j}\gamma_j,\qquad \gamma_j:=\frac{u_j^\star u_j}{2 t_1(A)}
$$
is a convex combination of elements of $W(A)$ and therefore $\chi\in W(A)$ itself. 
\end{proof}

\begin{example}
It is no longer true that when some $\chi_\pm$ coincides with an eigenvalue the numerical range is essentially reduced into a lower dimensional one. 
Consider
$$
A_4:=\zav{\begin{array}{cccc}
1-\imag & 0 & -2-\imag & 2+\imag \\
0 & -1-\imag & 1 &-\imag\\
0 & 0 & 1 & 2-2\imag \\
0 & 0 & 0 & -1 
\end{array}}.
$$
Clearly the eigenvalues of $A$ are $\pm 1$ and $\pm 1-\imag$. And little bit of calculation shows that $\chi_\pm=\pm 1-\imag$ since the equation (\ref{antipoleeq}) reads
$$
20(x^2+2\imag x-2)=0.
$$
Yet, the curve $\partial W(A)$ is not an ellipse with foci at $\pm 1$. To see that, observe
$$
W(\Im(A))=[-2.122191853\dots, 2.222569862\dots].
$$ 
Therefore $W(A)$ is not even symmetric with respect to $x$-axis. See Figure~\ref{4by4counterexample}.
\begin{figure}[ht] 
\centering 
\includegraphics[scale=0.5,trim=2cm 7cm 2cm 4cm,clip]{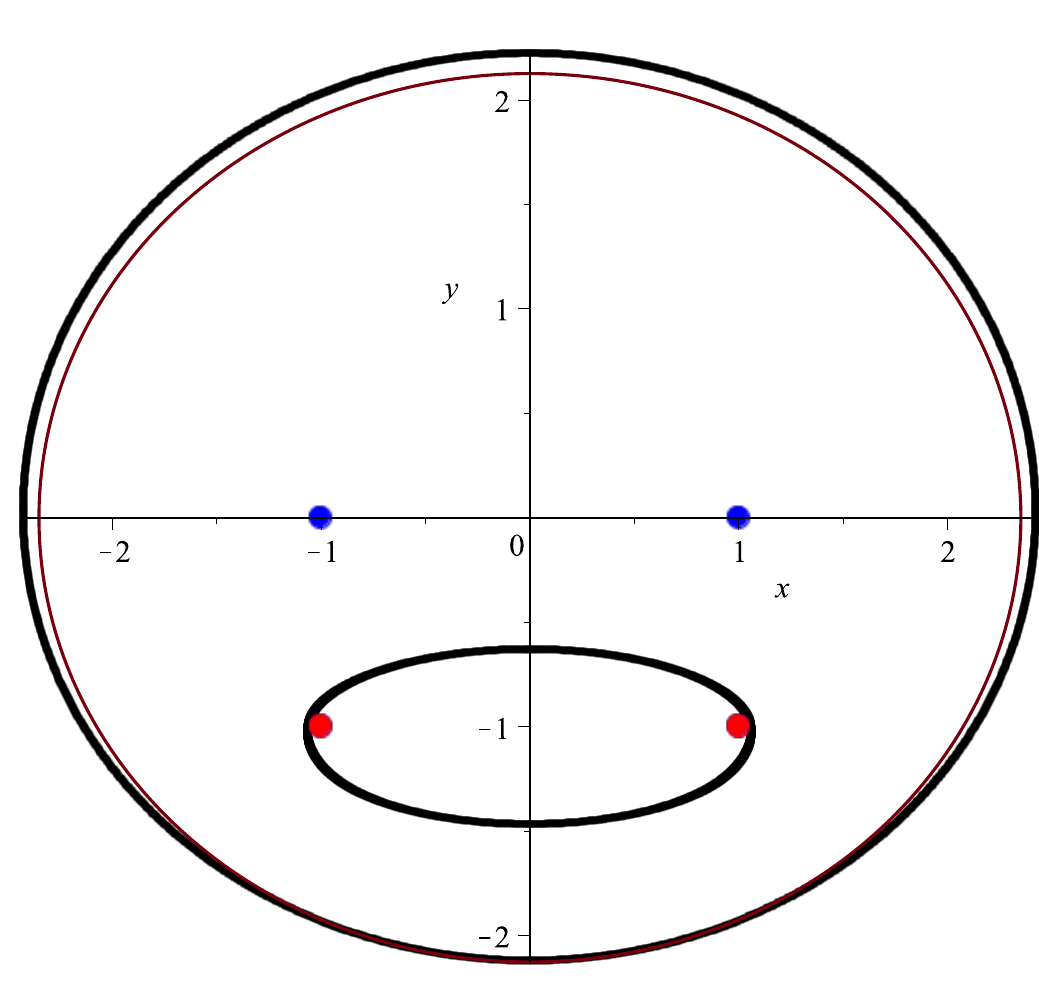}
\caption{The Kippenhahn curve of the matrix $A_4$ with eigenvalues at $\pm1$ whose remaining eigenvalues $-\imag\pm 1$ coincides with the numbers $\chi_\pm$. Trial ellipse with foci at $\pm 1$ shows that the numerical range does not reduce to an ellipse in this case.}
\label{4by4counterexample}
\end{figure}
\end{example}

\begin{definition}
Let $A\in\mathbb{C}^{4\times 4}$. In the notation of Corollary~\ref{4by4case} the \textit{inner conic} of $A$ is the singular solution of
\begin{equation}
t_1(A)p(\chi_+)p(\chi_-)=\frac{\gamma_2}{4}.
\end{equation}
\end{definition}
The analog of Corollary~\ref{innerlines3} is the following.
\begin{corollary}\label{innerlines4}
Let $A\in\mathbb{C}^{4\times 4}$. 
Any tangent of the inner conic of $A$ that passes through one of the eigenvalues is also tangent to the Kippenhahn curve.
\end{corollary} 
\begin{proof}
Obvious.
\end{proof}
\begin{example}
Consider
$$
A_3:=\left( \begin {array}{cccc} -{\frac {33}{10}}-{\frac {21\imag}{10}}&{
\frac {79}{30}}-{\frac {14\imag}{5}}&\frac95+{\frac {77\imag}{30}}&-{\frac {11
}{6}}+{\frac {19\imag}{6}}\\ \noalign{\medskip}0&{\frac {17}{15}}-{
\frac {89\imag}{30}}&{\frac {11}{6}}+{\frac {23\imag}{10}}&\frac75+{\frac {19
\imag}{30}}\\ \noalign{\medskip}0&0&{\frac {12}{5}}-\frac{\imag}{10}&{\frac {23}{10}
}-\frac{5\imag}{2}\\ \noalign{\medskip}0&0&0&{\frac {16}{5}}+{\frac {13\imag}{5}}
\end {array} \right).
$$
The implications of Corollary~\ref{innerlines4} are depicted in Figure~\ref{innerconicf}.
\begin{figure}[ht] 
\centering 
\includegraphics[scale=0.4,trim=2cm 3cm 2cm 7cm,clip]{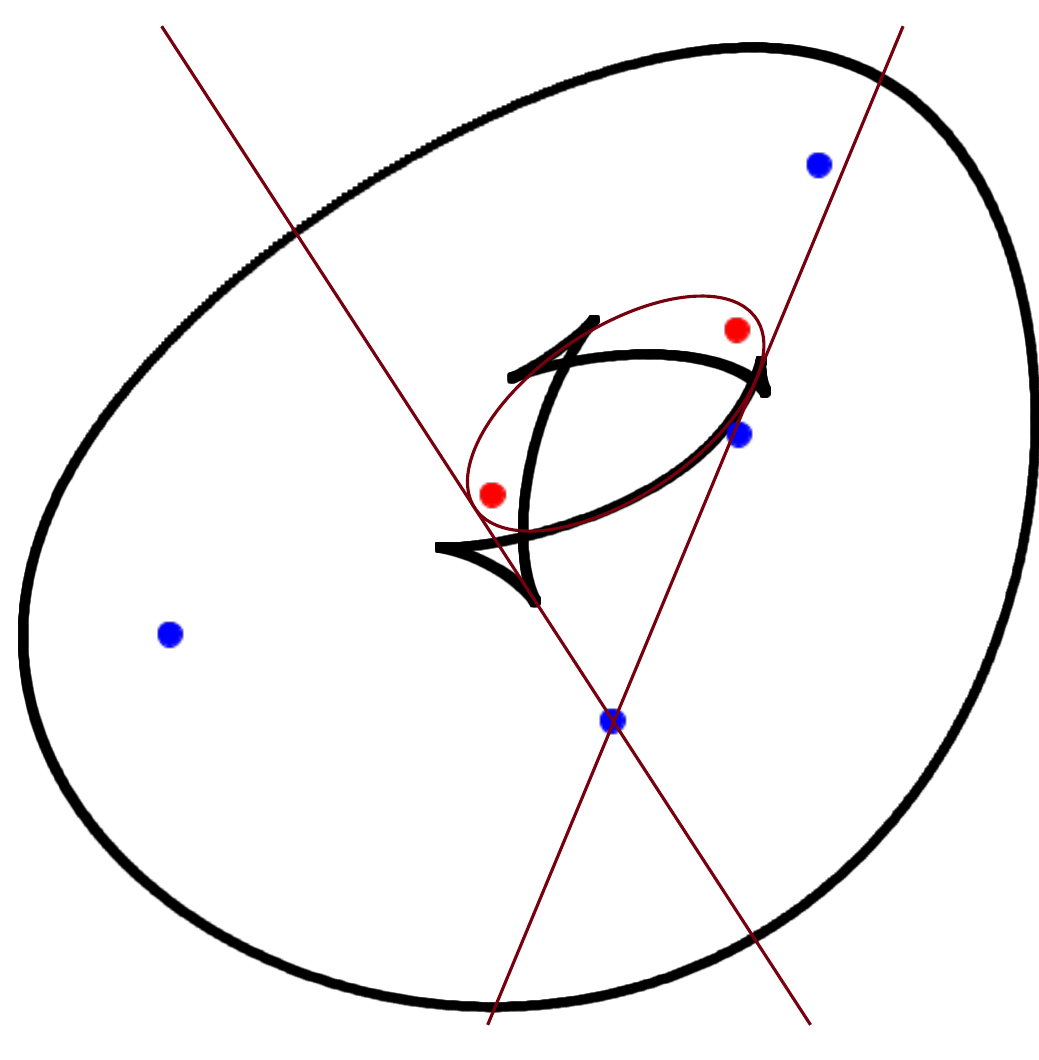}
\caption{The Kippenhahn curve of the matrix $A_3$ and its inner conic. The tangent lines of the inner conic passing through one eigenvalue are also tangents of the Kippenhahn curve.}
\label{innerconicf}
\end{figure}
\end{example}
\begin{example}
A geometrical construction of the inner conic can be obtained in some special cases. Consider
$$
A:=\zav{\begin{array}{cccc} 
\lambda_1 & a & 0 & 0 \\
0 &\lambda_2 & 0 & 0  \\
0& 0& \lambda_3 & b \\
0&0& 0& \lambda_4
\end{array}},\qquad a,b>0.
$$
It is obvious that the numerical range $W(A)$ is just the convex hull of two ellipses. Particularly
$$
\deter{p(A)}=\zav{p(\lambda_1)p(\lambda_2)-a^2}\zav{p(\lambda_3)p(\lambda_4)-b^2}
$$
From Corollary~\ref{4by4case} we also have
$$
\deter{p(A)}=p(\lambda_1)p(\lambda_2)p(\lambda_3)p(\lambda_4)-\frac{t_1(A)}{4}p(\chi_+)p(\chi_-)+\frac{\gamma_2}{16}.
$$
Assuming that the two ellipses are sufficiently distant form one another we can draw 8 lines:
\begin{align*}
p(\lambda_1)&=0, & p(\lambda_3)p(\lambda_4)&=b^2, &\text{two tangents of the second ellipse through $\lambda_1$}\\
p(\lambda_2)&=0, & p(\lambda_3)p(\lambda_4)&=b^2, &\text{two tangents of the second ellipse through $\lambda_2$}\\
p(\lambda_3)&=0, & p(\lambda_1)p(\lambda_2)&=a^2, &\text{two tangents of the first ellipse through $\lambda_3$}\\
p(\lambda_4)&=0, & p(\lambda_1)p(\lambda_2)&=a^2, &\text{two tangents of the first ellipse through $\lambda_4$}.
\end{align*}
All these lines are tangent to the Kippenhahn curve $\deter{p(A)}=0$ and to the inner conic:
$$
t_1(A)p(\chi_+)p(\chi_-)=\frac{\gamma_2}{4}.
$$
Since any conic section is determined uniquely by 5 tangents, this is more than enough to construct the inner conic. See Figure~\ref{twoellipsenumrange}. This example also show that the se numbers $\chi_\pm$ need not belong to the numerical range $\chi_\pm\not\in W(A)$.
\begin{figure}[ht] 
\centering 
\includegraphics[scale=0.75,trim=1cm 1cm 1cm 1cm,clip]{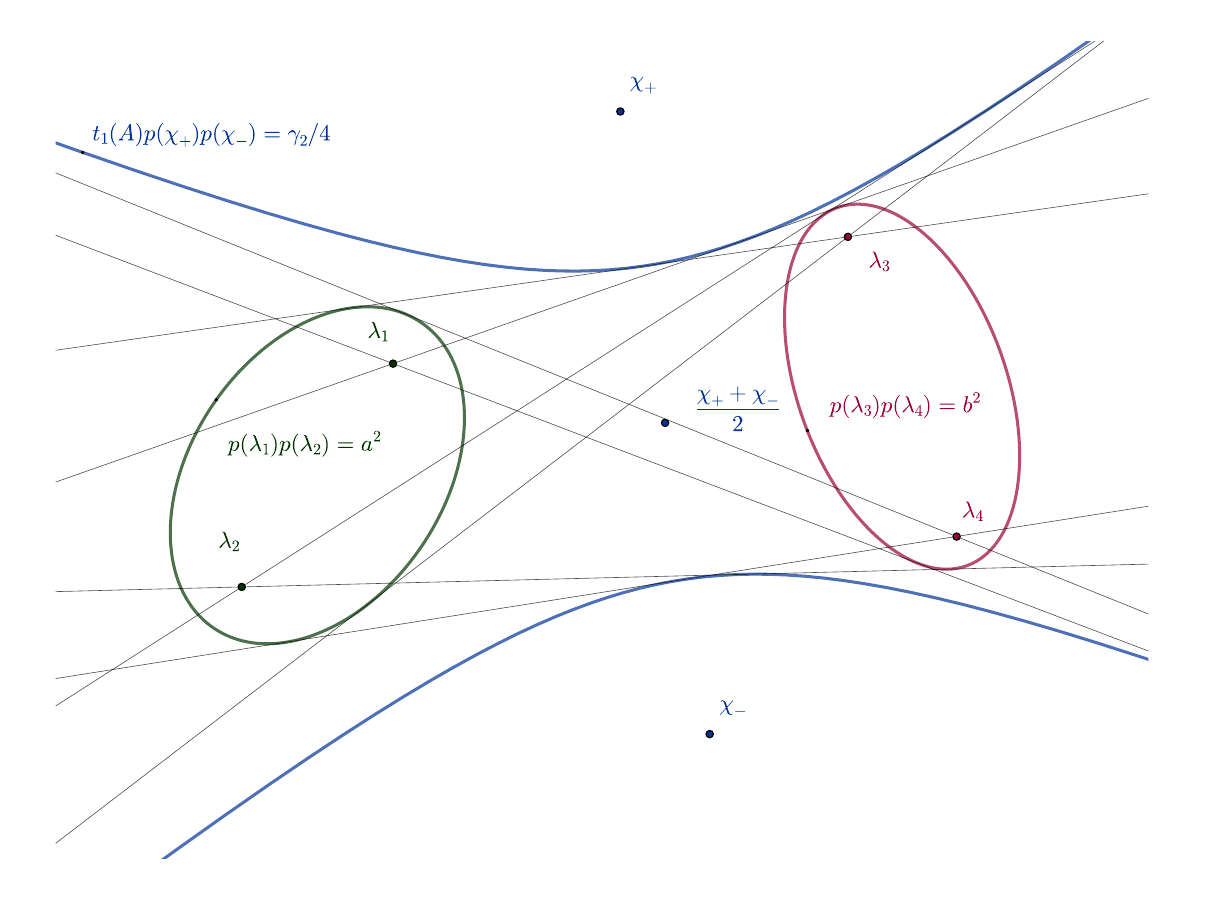}
\caption{Eight tangents of two ellipses determines uniquely the inner conic.}
\label{twoellipsenumrange}
\end{figure}

\end{example} 
\begin{example}
If the two ellipses in the previous example are not sufficiently apart we can still construct the inner ellipse using the following trick: We scale down the semi-minor axes $a,b$ by a factor of $\alpha>0$ and draw lines to these ellipses. Specifically
 \begin{align*}
p(\lambda_1)&=0, & p(\lambda_3)p(\lambda_4)&=\alpha b^2,\\
p(\lambda_2)&=0, & p(\lambda_3)p(\lambda_4)&=\alpha b^2,\\
p(\lambda_3)&=0, & p(\lambda_1)p(\lambda_2)&=\alpha a^2,\\
p(\lambda_4)&=0, & p(\lambda_1)p(\lambda_2)&=\alpha a^2.
\end{align*}
For sufficiently small $\alpha$ we have again 8 lines (unless $\lambda_1$ is not located in the complex segment $[\lambda_3,\lambda_4]$ and similarly for other eigenvalues). This time the lines are not tangent of the inner conic but uniquely determine a different conic
$$
t_1(A)p(\chi_+)p(\chi_-)=4a^2b^2(\alpha-1)+\frac{\gamma_2}{4},
$$
with the \textit{same} foci $\chi_\pm$. Once the numbers $\chi_\pm$ are obtained we can construct the inner ellipse computing its semi-minor axes $\gamma_2/(4t_1(A))$. See Figure~\ref{twoellipsenumrangea}.
\begin{figure}[ht] 
\centering 
\includegraphics[scale=0.75,trim=1cm 1cm 1cm 1cm,clip]{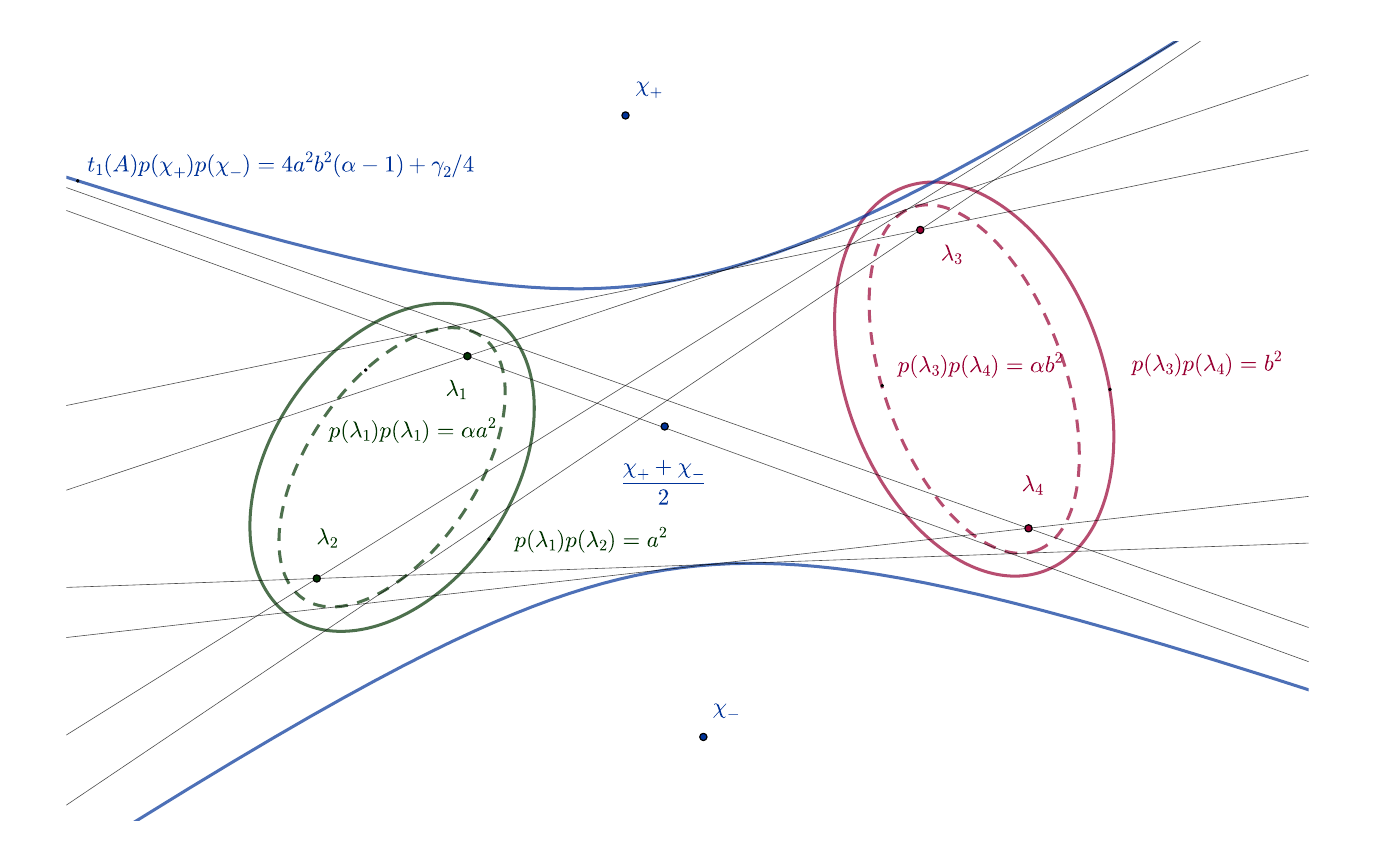}
\caption{Eight tangents of two scaled down ellipses can be also used.}
\label{twoellipsenumrangea}
\end{figure}
\end{example}

\subsection{Secondary values}
\begin{definition}
Let $A\in\mathbb{C}^n$, where $n\geq 3$. A \textit{secondary value of $A$} $x\in\mathbb{C}$ is a solution of 
\begin{equation}\label{antipoleeq}
\tr \zav{{\rm adj}(xI-A)(A^\star -\bar A)}=0.
\end{equation}
\end{definition}
\begin{theorem}\label{T2}Let $A\in\mathbb{C}^{n\times n}$ has the following block form
$$
A=\zav{\nadsebou{\lambda}{0}\nadsebou{b^\star}{D}},
$$
where $\lambda\in\mathbb{C}$ is an eigenvalue of $A$, $b\in \mathbb{C}^{(n-1)\times 1}$, $D\in\mathbb{C}^{(n-1)\times(n-1)}$.
\begin{enumerate}
\item If $b=0$ then $\lambda$ is also a secondary value of $A$.
\item If $\lambda\not\in W(D)$ and it is a secondary value, then $b=0$.
\end{enumerate}
\end{theorem}
\begin{remark}
In other words, the matrix $A$ decomposes into a direct sum of $\lambda$ and $D$ (i.e. $b=0$) only if $\lambda$ coincides with some secondary value. This condition is necessary but not sufficient. For that we need $\lambda$ being outside the numerical range of $D$. 
\end{remark}
\begin{proof}
The block form of $A$ implies
$$
\tr{\rm adj}\zav{x I-A}(A^\star-\bar A)=b^\star {\rm adj}\zav{x I- D} b+(x-\lambda) \tr{\rm adj}\zav{x I-D}(D^\star-\bar D).
$$
Therefore
$$
\tr{\rm adj}\zav{\lambda I-A}(A^\star-\bar A)=b^\star {\rm adj}\zav{\lambda I- D} b.
$$
For $b=0$ we thus have
$$
\tr{\rm adj}\zav{\lambda I-A}(A^\star-\bar A)=0,
$$
proving the first claim.

For the remaining claim, suppose $b\not=0$. Since $\lambda \not \in W(D)$ the eigenvalue $\lambda$ is simple and $\lambda I-D$ is an invertible matrix.

Denote $c:=(\lambda I-D)^{-1} b\not=0$ so that  $b=(\lambda I-D) c$ and $b^\star=c^\star \zav{\bar \lambda I-D^\star}$. We have
$$
0=\tr {\rm adj}\zav{\lambda I-A}(A^\star-\bar A)=b^\star {\rm adj}\zav{\lambda I- D} b
=\abs{\lambda I-D}c^\star \zav{\bar \lambda I-D^\star} c.
$$  
Discarding the non-zero determinant and taking the complex conjugate of both sides of the equation we obtain
$$
0=c^\star (\lambda I-D) c,
$$
rearranging
$$
\lambda=\frac{c^\star D c}{c^\star c}.
$$
Therefore $\lambda\in W(D)$. A contradiction. 
\end{proof}

\begin{proposition}The centroid of the secondary values is inside the numerical range. More precisely,
let $\chi_1,\dots, \chi_{n-2}$ be the secondary values of $A\in\mathbb{C}^n$, i.e. solutions of (\ref{antipoleeq}). Then
\begin{equation}
\chi:=\frac{\chi_1+\dots+\chi_{n-2}}{n-2}\in W(A).
\end{equation}
\end{proposition}
\begin{proof}
Note that by (\ref{anticharexplicit}) we have
$$
\chi=\frac{1}{n-2}\zav{\tr(A)-\frac{t_2(A)}{t_1(A)}}.
$$
We are going to show that $\chi$ is a convex combination of elements of $W(A)$. To that end assume that $A$ is upper triangular, i.e. 
$A_{j,k}=0$ for $j>k$. Denote its values $A_{j,k}=:a_{j,k}$ for $j<k$ and the eigenvalues by $\lambda_j:=A_{j,j}$.

For every triple of distinct natural numbers $j<k<l$ we define the following column vector $u_{jkl}\in \mathbb{C}^{n}$:
$$
\zav{u_{jkl}}_r=\left\{ \begin{array}{cc} 
a_{k,l} & r=j\\ 
-a_{j,l} & r=k\\ 
a_{j,k} & r=j\\
0 & \text{else} 
 \end{array} \right.
$$
There are exactly $\zav{\nadsebou{n}{3}}$ of such vectors each containing $3$ off-diagonal elements of $A$ out of $\zav{\nadsebou{n}{2}}$ possible.
Obviously
$$
u_{ijk}^\star u_{ijk}=|a_{jk}|^2+|a_{jl}|^2+|a_{kl}|^2.
$$
In the sum
$$
\sum_{j<k<l}u_{jkl}^\star u_{jkl}
$$
we have $3\zav{\nadsebou{n}{3}}$ summands and each off-diagonal element $|a_{j,k}|^2$ appears exactly $3\zav{\nadsebou{n}{3}}/\zav{\nadsebou{n}{2}}=(n-2)$ times. Therefore it holds
$$
\sum_{j<k<l}u_{jkl}^\star u_{jkl}=(n-2)t_1(A).
$$
Also
$$
u_{jkl}^\star A u_{jkl}=\lambda_j |a_{k,l}|^2+\lambda_k |a_{j,l}|^2+\lambda_{l} |a_{j,k}|^2-a_{j,k}a_{k,l}\overline{a_{j,l}}.
$$
Let us compute
\begin{align*}
\sum_{j<k<l}u_{jkl}^\star A u_{jkl}&=\sum_{j<k<l}\zav{\lambda_j |a_{k,l}|^2+\lambda_k |a_{j,l}|^2+\lambda_{l} |a_{j,k}|^2}-\sum_{j<k<l}a_{j,k}a_{k,l}\overline{a_{j,l}}\\
&=\zav{\sum_{j<k<l}+\sum_{k<j<l}+\sum_{k<l<j}}\lambda_j |a_{k,l}|^2-\sum_{j<k<l}a_{j,k}a_{k,l}\overline{a_{j,l}}\\
&=\zav{\sum_{j}\sum_{k<l}-\sum_{j=k<l}-\sum_{k<j=l}}\lambda_j |a_{k,l}|^2-\sum_{j<k<l}a_{j,k}a_{k,l}\overline{a_{j,l}}\\
&=\tr(A)t_1(A)-\sum_{k<l}\lambda_k |a_{k,l}|^2-\sum_{k<l}\lambda_l |a_{k,l}|^2-\sum_{j<k<l}a_{j,k}a_{k,l}\overline{a_{j,l}}\\
&=\tr(A)t_1(A)-\sum_{j\leq k\leq l}a_{j,k}a_{k,l}\overline{a_{j,l}}+\sum_{j=k= l}a_{j,k}a_{k,l}\overline{a_{j,l}}\\
&=\tr(A)t_1(A)-\tr\zav{A^2A^\star}+\tr \zav{A^2\bar A}\\
&=\tr(A)t_1(A)-t_2(A).
\end{align*}
Dividing bots sides by $(n-2)t_1(A)$ we obtain
$$
\chi=\sum_{j<k<l}\frac{u_{jkl}^\star A u_{jkl}}{u_{jkl}^\star u_{jkl}}\gamma_j,\qquad \gamma_j:=\frac{u_{jkl}^\star u_{jkl}}{(n-2)t_1(A)}.
$$
Since $\sum_j\gamma_j=1$ the number $\chi$ is a convex combination of elements of $W(A)$ and therefore $\chi\in W(A)$. 
\end{proof}

\section{Open problems}
\subsection*{Matrices with poles at infinity}
It seems plausible that there are matrices with greedy expansion of the form (\ref{covalpa}) that require one or more higher order poles $\lambda_j^{(k)}$, $k>2$ to be at infinity. But the author cannot find a specific example of this. The necessary condition is $\gamma_k=0$ for some $k\geq 2$ (in the notation of Theorem~\ref{maint}). 

The problem is that even $\gamma_2$ is, in general, an \textit{algebraic function } of the matrix entries (not merely polynomial as is the case for $\gamma_0$, $\gamma_1$) -- due to the interference of the secondary values -- and thus difficult to analyze.
 
\subsection*{Missing case of Theorem~\ref{T2}}
When an eigenvalue $\lambda$ is also a secondary value and furthermore it is outside $W(D)$ in the notation of Theorem~\ref{T2}, we know everything. What about the case $\lambda\in W(D)$? 

This easy to satisfy. Suppose
$$
\lambda=\frac{c^\star D c}{c^\star c},
$$
for some $c$, i.e. $\lambda\in W(D)$ and let $b=(\lambda I-D)c$. Denote
$$
A=\zav{\nadsebou{\lambda }{0}\nadsebou{b^\star}{D}}.
$$
Then $\lambda$ is both eigenvalue and secondary value of $A$, that is
$$
b^\star (\lambda I-D)^{-1}b=c^\star (\bar \lambda I-D^\star )c=\bar \lambda c^\star c-\overline{c^\star D c}=0,
$$
from the definition of $\lambda$. Yet
$$
\deter{p(A)}\not=p(\lambda)\deter{p(D)}.
$$
In fact
$$
\deter{p(A)}=p(\lambda)\deter{p(D)}-\frac14 b^\star{\rm adj}(p(D))b,
$$
and since the matrix $p(D)$ is positive definite (at least for $\partial W(A)$) -- and hence ${\rm adj}(p(d))$ is also positive definite -- the term $b^\star{\rm adj}(p(D))b>0$ cannot vanish.

Can we still draw some conclusions about numerical range of $A$ in this case?

\subsection*{Higher rank numerical ranges} The rank $k$ numerical range $\Lambda_k(A)$ is defined
$$
\Lambda_k(A):=\szav{\alpha\in\mathbb{C}: U^\star A U=\zav{\nadsebou{\alpha I_k}{*}\nadsebou{*}{*}}, \text{ for some unitary }U}.
$$
See e.g. \cite{rankkpaper1,rankkpaper2,rankkpaper3}. 
It is easy to see that $\Lambda_1(A)=W(A)$. The set $\Lambda_k(A)$ is also convex \cite{rankkpaper4}.
Equivalently \cite{rankkpaper5},
$$
\Lambda_k(A)=\szav{\alpha\in\mathbb{C}: e^{\imag \theta}\alpha+ e^{-\imag \theta}\bar \alpha\leq \lambda_k\zav{e^{\imag \theta} M+e^{-\imag \theta} M^\star}},
$$ 
where for a Hermitian matrix $H$ the symbol $\lambda_k(H)$ coresspond to the $k$-th largest eigenvalue of $H$.

What is a coval representation of $\partial \Lambda_k(A)$?

\subsection*{Properties of a ''2-normal'' matrix}
For a normal matrix $A$ we have 
$$
\deter{p(A)}=\prod_{j=1}^n p(\lambda_j),
$$
i.e. in our greedy expansion formula (\ref{covalpa}) only single term survives.

What are the properties of a matrix $A$ for which \textit{two} terms survive? 
\begin{definition}
We will call a matrix $A\in\mathbb{C}^{n\times n}$ \textit{2-normal} if
\begin{equation}\label{2normaldef}
\deter{p(A)}=\prod_{j=1}^n p(\lambda_j)-\frac{t_1(A)}{4}\prod_{j=1}^{n-2}p(\chi_j).
\end{equation}
\end{definition}
Two-normal matrices are very nice, geometrically, since their numerical range can be described using eigenvalues and the secondary values only.
Furthermore, the secondary values $\chi_j$, in this case, are truly behaving like an ``antiparticles'' to eigenvalues exactly the same way as in dimension 3 -- when one coincides with the other they ``annihilate'' and the numerical range is reduced to a lower dimensional one (plus a possible point). No further condition needed.

There is a simple counting argument from which we can get a rough idea how many 2-normal matrices there are. An upper triangular $n\times n$ complex matrix with fixed eigenvalues has precisely $n(n-1)/2$ free complex parameters, or $n^2-n$ free real parameters.

From (\ref{2normaldef}) we know that any line connecting an eigenvalue $\lambda$ with some secondary value $\chi$ is a solution to $\deter{p(A)}=0$. Therefore computing the tangential angle $\theta$ and the pedal coordinate $p$ of such line, we have
$$
\deter{p I-\frac{e^{\imag\theta}}{2\imag} A^\star+\frac{e^{-\imag \theta}}{2\imag} A}=0.
$$
A single (non-linear) equation in parameters of $A$.
There are $n(n-2) $ such lines, therefore we have $n(n-2)$ equations. Therefore -- roughly -- the space of 2-normal matrices should have at most $n$ real parameters.
\bigskip

In the literature, there exists so-called \textit{binormal matrices} -- see \cite{Ikramov}. A matrix $A$ is binormal if $A^\star A$, $AA^\star$  commute. 
The anonymous referee has suggested the following example of a 2-normal matrix which is not binormal. Consider
$$
A(b):=\zav{\begin{array}{cccc}
0 & 1 & b & 1 \\
0 & 0 & 1 & b \\
0 & 0 & 0 & 1 \\
0& 0 & 0 & 0
\end{array}}.
$$
Easy calculation shows that 
$$
\det{p(A(b))}=p^4-\frac{2|b|^2+4}{4}p(\chi_+)p(\chi_-)-\frac{\gamma_2}{16},
$$
where
$$
\chi_\pm:=\frac{-(b+\bar b)\pm \sqrt{b^2+{\bar b}^2-4}}{2|b|^2+4},
$$
and, crucially,
$$
\gamma_2=4-(b+\bar b)^2+|b|^4=b_2^4+2b_1^2b_2^2+(b_1^2-2)^2,\qquad b=b_1+\imag b_2.
$$
We can see that $\gamma_2\geq 0$ and $\gamma_2=0$ if and only if $b=\pm\sqrt{2}$. Interestingly, these values of $b$ are precisely those for which the secondary values coincides, i.e. $\chi_+=\chi_{-}$. This might be a coincidence, though.

\subsection*{False conjecture}
A natural conjecture one might have is that the following holds:

For all $A\in \mathbb{C}^{n\times n}$ there exists $B\in \mathbb{C}^{(n-2)\times (n-2)}$ such that
\begin{align}\label{falsecon}
\deter{p(A)}\stackrel{?}{=}\prod_{j=1}^n p(\lambda_j)-\frac{t_1(A)}{4}\deter{p(B)},
\end{align}
where $\lambda_j$ are eigenvalues of $A$.

Unfortunately, this is not the case in general. Consider the following $7\times 7$ matrix
$$
A:= \left( \begin {array}{ccccccc} -{\frac {27}{10}}+{\frac {46\,i}{15}}&
{\frac {14}{15}}+{\frac {89\,i}{30}}&-{\frac {67}{30}}+{\frac {27\,i}{
10}}&{\frac {23}{30}}+{\frac {17\,i}{30}}&{\frac {13}{5}}+5/6\,i&-5/6+
{\frac {19\,i}{30}}&-{\frac {17}{30}}-{\frac {17\,i}{6}}
\\ \noalign{\medskip}0&-9/5-{\frac {17\,i}{6}}&7/6-{\frac {43\,i}{30}}
&-{\frac {23}{15}}+{\frac {29\,i}{15}}&2/3+{\frac {34\,i}{15}}&{\frac 
{17}{10}}-3/2\,i&-2+{\frac {83\,i}{30}}\\ \noalign{\medskip}0&0&8/5+{
\frac {19\,i}{10}}&{\frac {16}{15}}-{\frac {19\,i}{30}}&{\frac {13}{30
}}-{\frac {17\,i}{6}}&-{\frac {37}{15}}+5/2\,i&{\frac {13}{15}}-{
\frac {79\,i}{30}}\\ \noalign{\medskip}0&0&0&-7/6-{\frac {12\,i}{5}}&{
\frac {13}{10}}-9/5\,i&2/3+{\frac {47\,i}{30}}&8/3-i/6
\\ \noalign{\medskip}0&0&0&0&7/6-i/3&{\frac {27}{10}}+3/2\,i&3/10+3/10
\,i\\ \noalign{\medskip}0&0&0&0&0&-{\frac {29}{30}}-{\frac {79\,i}{30}
}&-{\frac {23}{15}}-{\frac {41\,i}{30}}\\ \noalign{\medskip}0&0&0&0&0&0
&-1/15+3\,i\end {array} \right),
$$
and let
\begin{align*}
f(\theta)&:=|p(A)|-p(\lambda_1)\cdots p(\lambda_5)=\left|pI-\frac{e^{\imag \theta}}{2\imag}A^\star+\frac{e^{-\imag \theta}}{2\imag}A\right|-\left|pI-\frac{e^{\imag \theta}}{2\imag}\bar A+\frac{e^{-\imag \theta}}{2\imag}A\right|.
\end{align*}
We want to find $5\times 5$ matrix $B$ such that
$$
f(\theta)=\frac{t_1(A)}{4}\left|pI-\frac{e^{\imag \theta}}{2\imag}B^\star+\frac{e^{-\imag \theta}}{2\imag}B\right|,
$$ 
i.e. (up to a factor) $f(\theta)$ is a determinant of a Hermitian symmetric matrix whose eigenvalues are therefore all real. Thus all 5 roots of $f(\theta)$ (as a polynomial in $p$) must be real for all $\theta$.
But
\begin{align*}
f(-\pi/2)&=-{\frac {1629\,{p}^{5}}{50}}+{\frac {525707\,{p}^{4}}{6000}}+{\frac {
154689277\,{p}^{3}}{540000}}-{\frac {136887958013\,{p}^{2}}{194400000}
}-{\frac {29907897828361\,p}{46656000000}}+{\frac {393834343927013}{
279936000000}},\\
\intertext{has only three real roots.} 
\end{align*}
Still, the equality (\ref{falsecon}) might be true for \textit{some} matrices. Promising candidate are cyclic weighted shift matrices \cite{cyclicpaper1}. Or so-called "Unitary bordering matrices" \cite{UBM1,UBM2,UBM3}.

The question of existence of such a matrix $B$ is also equivalent to the condition: the Kippenhahn
polynomial $\deter{tI+\Re(A)v+\Im(A)w}$ is hyperbolic with respect to the vector $(t,u,w)=(1, 0, 0)$. See \cite{hyperbolicpaper} for the details and \cite[page 130]{hyperbolicpaper2} for a useful fact that a factor of a hyperbolic form is also hyperbolic.

\section{Acknowledgement}
The author would like to thank Ilya Spitkovsky and Piotr Pikul for valuable comments. The author is also indebted to the anonymous reviewer for many useful insights and observations.

Research supported by GA\v{C}R grant no. 25-18042S and RVO funding for I\v{C}O 67985840

\end{document}